\documentclass[11pt,twoside, a4paper]{article}
\usepackage[utf8]{inputenc}
\usepackage[T1]{fontenc}
\usepackage{geometry}
\usepackage{amsmath}
\usepackage{amsfonts}
\usepackage{amssymb}
\usepackage{tikz}
\usepackage{anyfontsize}
\usepackage{graphicx}
\usepackage{rotating}
\usepackage{amsthm}
\usepackage{yfonts}
\usepackage{pdfpages}
\usepackage{float}
\usepackage[title]{appendix}
\usepackage[]{biblatex}
\usepackage{hyperref}
\hypersetup{colorlinks=true, citecolor=red}
\usepackage{cleveref}
\usepackage{lastpage}
\usepackage{fancyhdr}
\newcommand{\N}{\mathbb{N}}
 
\newcommand{\R}{\mathbb{R}}
\newcommand{\Z}{\mathbb{Z}}

\renewcommand{\epsilon}{\varepsilon}
\renewcommand{\H}{\mathbb{H}}
\newtheorem{thm}{Theorem}
\newtheorem{prop}[thm]{Proposition}
\newtheorem{corollaire}[thm]{Corollary}

\theoremstyle{definition}

\newtheorem{defn}[thm]{Definition}
\newtheorem{rem}[thm]{Remark}
\newtheorem{expl}[thm]{Example}
   
\newtheorem{question}[thm]{Question}

\usepackage{color}

\title{The Steklov problem on triangle-tiling graphs in the hyperbolic plane} 

\author{Léonard Tschanz 
\vspace{1cm} \\
Institut de Mathématiques, Neuchâtel University, Neuchâtel, Switzerland \\
leonard.tschanz@unine.ch \\
ORCID: 0000-0003-2097-6759
}

\date{}

\begin{document}
\maketitle

\begin{abstract}
\noindent We introduce a graph $\Gamma$ which is roughly isometric to the hyperbolic plane and we study the Steklov eigenvalues of a subgraph with boundary $\Omega$ of $\Gamma$. For $(\Omega_l)_{l\geq 1}$ a sequence of subraphs of $\Gamma$ such that $|\Omega_l| \longrightarrow \infty$, we prove that for each $k \in \N$, the $k^{\mbox{th}}$ eigenvalue tends to $0$ proportionally to $1/|B_l|$. The idea of the proof consists in finding a bounded domain $N$ of the hyperbolic plane which is roughly isometric to $\Omega$, giving an upper bound for the Steklov eigenvalues of $N$ and transferring this bound to $\Omega$ via a process called discretization.
\medskip

\noindent \textbf{Keywords:} Spectral geometry, Steklov problem, graphs with boundary, discrete Steklov problem.
\medskip

\end{abstract}

\section{Introduction} \label{section : 1}

Let  $(M, g)$ be a smooth connected compact Riemannian manifold of dimension $n \geq 2$ with smooth boundary $\partial M$. The Steklov problem on  $(M, g)$ consists in finding all  $\sigma \in \R$ such that there exists a non-zero harmonic function $f : M \longrightarrow \R$ satisfying $\frac{\partial f}{\partial \nu} = \sigma f$ on $ \partial M$, where $\frac{\partial }{\partial \nu}$ denotes the outward-pointing normal derivative on $\partial M$.
\medskip

Such a $\sigma$ is called a Steklov eigenvalue of $M$ and a corresponding $f$ is called a Steklov eigenfunction. The (ordered) set of eigenvalues is called the Steklov spectrum of $(M, g)$.
\medskip

It is well known that the Steklov spectrum of $M$ forms a discrete sequence 
\begin{align*}
0 = \sigma_0 < \sigma_1 \leq \sigma_2 \leq \ldots \nearrow \infty,
\end{align*}
where each eigenvalue is repeated with multiplicity.
\medskip

There exists a discrete analog to the Steklov problem, which is called the discrete Steklov problem and which is defined on graphs with boundary. Let us begin by defining it.

\begin{defn}
A graph with boundary is a triplet $(\bar{\Omega}, E', B)$, where $(\bar{\Omega}, E')$ is a simple connected undirected graph and $B \subset \bar{\Omega}$ is a non-empty set of vertices, called the boundary. The set $B^c$ is called the interior of the graph.
\end{defn}

In this paper, all graphs will always be simple connected and undirected.
\medskip

For $v, w \in \bar{\Omega}$, we write $v \sim w$ when $v$ is adjacent to $w$. For $A \subset \bar{\Omega}$, we write $|A|$ the cardinality of $A$, which is the number of vertices contained in $A$. For the purpose of this article, all graphs with  boundary are finite. We denote by $\R^{\bar{\Omega}}$ the space of all functions $u : \bar{\Omega} \longrightarrow \R$, which is isomorphic to the Euclidean space of dimension $|\bar{\Omega}|$. Similarly, we denote by $\R^B$ the space of functions $u : B \longrightarrow \R$, which is the Euclidean space of dimension $|B|$.
\medskip

We can now introduce the discrete Laplacian operator $\Delta : \R^{\bar{\Omega}} \longrightarrow \R^{\bar{\Omega}}$, defined by
\begin{align*}
\Delta u : \bar{\Omega} & \longrightarrow \R \\
            v & \longmapsto \Delta u(v) = \sum_{w \sim v}(u(v)-u(w)).
\end{align*}
The normal derivative $\frac{\partial}{\partial \nu} : \R^{\bar{\Omega}} \longrightarrow \R^B$ is defined by
\begin{align*}
\frac{\partial u}{\partial \nu} : B & \longrightarrow \R \\
                                 v & \longmapsto \frac{\partial u}{\partial\nu}(v) = \sum_{w \sim v}(u(v)-u(w)).
\end{align*}

As one can see, the normal derivative coincides with the restriction of the Laplacian to the boundary. Although this choice may seem strange, it is shown in \cite{CGR} that it leads to interesting links between the Steklov spectrum of a manifold and the Steklov spectrum of a graph with boundary which \textit{looks like} the manifold, see \parencite[Theorem 3]{CGR} for more information about what \textit{looks like} means in this context.

\begin{defn}
The discrete Steklov problem on a finite graph with boundary $(\bar{\Omega}, E', B)$ consists in finding all $\sigma \in \R$ such that there exists a non-zero function  $u \in \R^{\bar{\Omega}}$ such that 
\begin{align*}
    \left\{
    \begin{array}{l}
         \Delta u(v) = 0 \mbox{ if } v \in \Omega  \\
         \frac{\partial}{\partial \nu}u(v) = \sigma u(v) \mbox{ if } v \in B.
    \end{array}
    \right.
\end{align*}

\end{defn}
Such a $\sigma$ is called a Steklov eigenvalue  and a corresponding $u$ is called a Steklov eigenfunction of $(\bar{\Omega}, E', B)$. As said in  \cite{Per1}, the Steklov spectrum of a graph with boundary $(\bar{\Omega}, E', B)$ forms a sequence as follows:
\begin{align*}
0 = \sigma_0 < \sigma_1 \leq \sigma_2 \leq \ldots \leq \sigma_{|B|-1}.
\end{align*}

This problem has recently received a particular attention, one can cite for instance \cite{HH, HHW, Per2, Per1}. An investigation has been made by Colbois, Girouard and Raveendran in \cite{CGR}, allowing us to understand some spectral links between the Steklov problem on a manifold and the discrete Steklov problem of a graph associated to this manifold. These links will be very useful in this paper. The main problem that we will have to face is to place ourselves in the hypotheses of Theorem $3$ of \cite{CGR}, in order to use it to our advantage.
\medskip

Among other things, a question that has been studied by some authors is that of providing an upper bound for the first - and then for the $k^{\mbox{th}}$ - eigenvalue of some particular graphs with boundary.  These particular graphs that have been studied are those called \textit{subgraphs} of an (infinite) \textit{host graph}.
A subgraph of a host graph  can be interpreted as the discrete analog of a bounded domain in a manifold.
Let us define what it is exactly.

\begin{defn} \label{def : subgraph}
Let $\Gamma = (V,E)$ be a graph and let $\Omega \subset V$ be a finite subset of vertices connected for $\Gamma$, i.e for each $v, w \in \Omega$, there exist $l \in \N$ and $v_0 = v, v_1, \ldots, v_l = w \in \Omega$ satisfying $\{v_i, v_{i+1}\} \in E$ for all $i = 0, \ldots, l-1$. The graph with boundary $(\bar{\Omega}, E', B)$ induced by $\Omega$ is defined as follows:
\begin{itemize}
\item $ B = \{w  \in V \backslash \Omega : \exists \; v \in \Omega  \; \mbox{such that} \; \{v,w\} \in E\}$;
\item $\bar{\Omega} = \Omega \cup B$;
\item $E' = \{ \{v,w\} \in E : v \in \Omega  , w \in \bar{\Omega} \}$.
\end{itemize}
Such a graph with boundary is simply denoted $\Omega$ and is called subgraph of $\Gamma$. The set of vertices $B$ is the boundary of the subgraph. We refer to $\Gamma$ as the host graph of $\Omega$.
\end{defn}

Some interesting results have recently been discovered, providing us with bounds for the eigenvalues, depending on the host graph $\Gamma$. 
A first result, due to Han and Hua, is the following:
\begin{thm}[Theorem $1.2$ in \cite{HH}] \label{corollaire : sigma_1 HH}
Let $\Z^d$ be the integer lattice of dimension $d$. Let $\Omega$ be a subgraph of $\Z^d$. Then we have
\begin{align*}
\sum_{l=1}^{d} \frac{1}{\sigma_l(\Omega)} \geq C' \cdot |\Omega|^{\frac{1}{d}} - \frac{C''}{|\Omega|},
\end{align*}
where $C' = (64 d^3 \omega_d^{\frac{1}{d}})^{-1}, C''=\frac{1}{32d}$ and $\omega_d$ is the volume of the unit ball in $\R^d$. 
\end{thm}

Another investigation gives some control over  the spectrum of a subraph of a \textit{Cayley graph}. We recall that, given a finitely generated group $G$ and a finite generating subset $S$ of $G$, one can define a graph, called Cayley graph and denoted $Cay(G,S)$. If $G$ is infinite, then so is $Cay(G,S)$ and we can use it as a host graph. The result provided by Perrin is the following:

\begin{thm}[Corollary $1$ in \cite{Per2}] \label{corollaire : sigma_1 HP}
Let $\Gamma = (V,E)$ be a Cayley graph with polynomial growth of order $d \geq 2$. There exists $\tilde{C}(\Gamma) >0$ such that for any finite subgraph $\Omega$ of $\Gamma$, we have
\begin{align*}
\sigma_1(\Omega) \leq \tilde{C}(\Gamma) \cdot \frac{1}{|B|^{\frac{1}{d-1}}}.
\end{align*}
\end{thm}

This theorem is way more general about the class of host graph $\Gamma$ but provides us control over the first non-trivial eigenvalue only, see \cite{Per2} for details. We gave an extension to this result in a precedent article:

\begin{thm}[Theorem $5$ in \cite{T}] \label{thm : moi}
Let $\Gamma = Cay(G,S)$ be a polynomial growth Cayley graph of order $d \geq 2$. Let $\Omega$ be a subgraph of  $\Gamma$. Then there exists a constant $\bar{C}(\Gamma) >0$ such that for all $k < |B|$,
\begin{align*}
\sigma_k(\Omega) \leq \bar{C}(\Gamma) \cdot \frac{1}{|B|^{\frac{1}{d-1}}} \cdot k^{\frac{d+2}{d}}.
\end{align*}
\end{thm}

As a corollary, we have: 

\begin{corollaire}[Corollary $6$ in \cite{T}] \label{corollaire : sigma_k}
Let $\Gamma$ be a polynomial growth Cayley graph of order $d \geq 2$ and  $(\Omega_l)_{l=1}^\infty$ be a sequence of subgraphs of $\Gamma$ such that $|\Omega_l| \underset{l \to \infty}{\longrightarrow} \infty$. Fix $k \in \N$. Then we have
\begin{align*}
\sigma_k(\Omega_l) \underset{l \to \infty}{\longrightarrow} 0.
\end{align*}
\end{corollaire}

All these theorems follow from the investigation upon one class of host graphs $\Gamma$, which are Cayley graphs of polynomial growth groups. This consideration leads to a natural question: 
\begin{center}
    \textit{What can we say about the eigenvalues of subgraphs of a host graph $\Gamma$, whose growth rate is more than polynomial?}
\end{center}

A first class of graphs  we can think of is that of trees. In \cite{He-Hua}, the authors find upper bounds for the eigenvalues of a finite tree. Their investigations lead to the following result:

\begin{thm}[Theorem $1.1$ and $1.5$ in \cite{He-Hua}] \label{thm : HH trees}
Let $\mathcal{T}$ be a finite tree with (uniformly) bounded degree $D$. Let $B$ be the boundary of the tree, i.e the set of vertices of degree one. Then we have
\begin{align*}
\sigma_1  \leq \frac{4 (D-1)}{|B|}.
\end{align*}
Higher Steklov eigenvalues are bounded as well: for all $k = 2, \ldots, |B|-1$, we have
\begin{align*}
\sigma_k & \leq \frac{8(D-1)^2(k-1)}{|B|}.
\end{align*}
\end{thm}

As stated by Remark $1.7$ of \cite{He-Hua}, we can consider as the host graph $\Gamma$ the Cayley graph of a free group and use this result to estimate the Steklov eigenvalues of a subgraph $\Omega$ of $\Gamma$. Since the growth rate of such a host graph is exponential, we now have  a completely new class of host graphs for which we can estimate their subgraphs eigenvalues.

\medskip

This paper's objective is to study the subgraphs's eigenvalues of a host graph $\Gamma$ which is roughly isometric to the hyperbolic plane $\H^2$ (see Definition \ref{def : hyperbolic plane}). The hyperbolic plane is a Cartan-Hadamard manifold of constant sectional curvature $-1$. Then $\Gamma$ can be seen as a discrete analog of such a manifold. Because of its relation with $\H^2$, the growth rate of $\Gamma$ is exponential, and then $\Gamma$ does not enter the class of host graphs of Theorems \ref{corollaire : sigma_1 HH}, \ref{corollaire : sigma_1 HP} and \ref{thm : moi}. 

Despite a growth rate identical to that of the trees, the structure of $\Gamma$ is very different from the latter, because of its connection with $\H^2$. Therefore, the method we will use to obtain upper bounds has nothing to do with the one used in \cite{He-Hua}. Indeed, He and Hua were able to work directly on the trees and use the great ease of disconnection of the trees as a tool to obtain the bounds of Theorem \ref{thm : HH trees}, while on our side we will use the proximity between $\Gamma$ and $\H^2$ to obtain upper bounds. 
\medskip

There are many graphs which are roughly isometric to the hyperbolic plane. This paper will focus on a particular class of such graphs, coming from a tiling of $\H^2$ associated with a triangle group. We shall refer to such a graph as \textit{triangle-tiling graph}. 

Triangle groups are part of the Coxeter groups, which can be seen as groups generated by reflections. These groups have been studied by many authors, see for instance \cite{Bou, Hil, Hum}. Triangle groups are Coxeter groups with three generators, that can be regarded as reflections through the sides of a triangle. They lead to many beautiful geometric constructions and tiling, see \cite{Bea, CBG, Mag, Wiki_pavage_triangulaire}.  We will recall in Sect. \ref{section : 2} hereafter the notions that are required for the understanding of the paper. 
\medskip

Our main result is the following:

\begin{thm} \label{thm : principal}
Let $\Gamma$ be a triangle-tiling graph. Then there exists a constant $C = C(\Gamma) >0$ such that for all subgraph $\Omega$ of $\Gamma$ and all $k < |B|$, we have
\begin{align*}
\sigma_k(\Omega) \leq C(\Gamma) \cdot \frac{1}{|B|} \cdot k^{2}.
\end{align*}
\end{thm}

As we will see in Sect. \ref{section : 2}, the host graph $\Gamma$ is defined from the choice of three integers. As a consequence, we will see that there are infinitely many triangle-tiling graphs.

As a corollary, we obtain the interesting fact:

\begin{corollaire} \label{cor : zero}
Let $(\Omega_l)_{l \geq 1}$ be a family of subgraphs  of $\Gamma$ such that $|\Omega_l| \underset{l \to \infty}{\longrightarrow} \infty$. Then for all $k \in \N$ fixed,
\begin{align*}
\sigma_k(\Omega_l) \underset{l \to \infty}{\longrightarrow} 0.
\end{align*}
\end{corollaire}
The number $\sigma_k(\Omega_l)$ is of course defined if and only if $|B_l| < k$. This condition is satisfied for $l$ big enough thanks to the assumption that $|\Omega_l| \longrightarrow \infty$. 

\medskip

Our approach is sketched this way: we define a  triangle-tiling graph $\Gamma$ that we use as a host graph and show that it is roughly isometric to $\H^2$ (see Definition \ref{def : quasi isom}). Thanks to the rough isometry, we can naturally associate to a subgraph $\Omega$ of $\Gamma$ a bounded domain $N$ of $\H^2$, whose boundary will be denoted $\Sigma$. We can then use results from \cite{CEG1} to give upper bounds for $\sigma_k(N)$.

Once this task is completed we use the work of Colbois et al. presented in \cite{CGR} in order to discretize a Riemannian manifold with boundary $(N, g')$, obtained as a deformation of the domain $N$ (this deformation is necessary since we have to satisfy the assumptions of \parencite[Theorem 3]{CGR}). This discretization will give us a path linking the eigenvalues of $N$ and the ones of $\Omega$, which will allow us to conclude.
\medskip


Our strategy can be summed up in the diagram below:

\begin{center}
\begin{tikzpicture}
     \node (1) at (-10,-1) {$\Gamma$} ;
     \node (2) at (-4,-1) {$\Omega$} ;
     \node (3) at (-10,-4) {$\H^2$} ; 
     \node (4) at (-4,-4) {$(N, g)$} ;
     \node (5) at (2,-4) {$(N, g')$} ;
     \node (6) at (2,-1) {$(\bar{V}, \bar{E}, V_\Sigma)$} ;
     
     \draw[<->] (1) to node[pos=0.5]{\; roughly isometric}  (3) ;
     \draw[->] (1) to node[pos=0.5, above]{subgraph} (2) ;
     \draw[->] (3) to node[pos=0.5, above]{domain} (4) ;
     \draw[<->] (2) to node[pos=0.5]{structure preserved} (4) ;
     \draw[->] (4) to node[pos=0.5, above]{change of metric}  (5) ;
     \draw[->] (5) to node[pos=0.5]{discretization} (6) ;
     \draw[<->] (2) to node[pos=0.5, above]{roughly isometric}  (6) ;

\end{tikzpicture}

\end{center}
Here, by \textit{structure preserved}, we mean that the structural information of the subgraph $\Omega$ can be read in the domain $N$, see the rest of the paper for more details. Moreover, in the diagram, $P \longleftrightarrow Q$ reflects the idea that $P$ is in some sense analog to $Q$, and $P \longrightarrow Q$ reflects the idea that $Q$ is obtained from $P$. More details are given in the rest of the paper.
\medskip

Our result holds for subgraphs of any triangle-tiling graph. However, there exist many other graphs that are roughly isometric to the hyperbolic plane, and that we could use as  host graphs. This remark naturally leads to many interesting interrogations, that we will consider and develop in Sect. \ref{sect : interrogation}. In particular, one may ask if the result is still true when using other host graphs roughly isometric to $\H^2$. This leads to the following open question (\Cref{quest : autre graphe}):
\begin{center}
   \textit{If $\Gamma$ is any graph roughly isometric to the hyperbolic plane, is there a constant $C = C(\Gamma)$ such that a bound as in Theorem $\ref{thm : principal}$ exists?}
\end{center}
Moreover, if $(\Omega_l)_{l\ge 1}$ is a sequence of subgraphs such that $|\Omega_l| \underset{l \to \infty}{\longrightarrow} \infty$, then in many cases (\Cref{corollaire : sigma_k}, \Cref{cor : zero}, \parencite[Corollary 1.4]{He-Hua}) the behaviour of $\sigma_1(\Omega_l)$ satisfy $\sigma_1(\Omega_l)  \underset{l \to \infty}{\longrightarrow} 0$. However, that is not always true, see \parencite[Example 3.7]{HH}. One may ask if the property is preserved under rough isometry (\Cref{question : comportement}): 
\begin{center}
    \textit{Let $\Gamma_1, \Gamma_2$ be two roughly isometric host graphs. Let us assume that in $\Gamma_1$, each sequence of subgraphs $(\Omega_l)_{l\ge 1}$ such that  $|\Omega_l| \underset{l \to \infty}{\longrightarrow} \infty$ satisfies $\sigma_1(\Omega_l)  \underset{l \to \infty}{\longrightarrow} 0$. Does $\Gamma_2$ also have this property?} 
\end{center}
As said before, these interrogations, and other (including some about higher dimensional constructions), will be asked  in Sect. \ref{sect : interrogation}.

\medskip

\textbf{Notation.} Throughout this paper, we shall work on graphs, on domains of $\H^2$ and on a manifold obtained from the domains. As stated before, the host graph will be denoted $\Gamma = (V, E)$. A subgraph of $\Gamma$ is denoted $\Omega$, while $N$ and $\tilde{N}$ are used to speak about domains of $\H^2$. We use $g$ to denote the metric of $\H^2$ and $g'$ the one of the manifold; hence $(N, g')$ is the notation we will use to speak about the manifold. A discretization of the manifold will be called $(\tilde{V}, \tilde{E}, V_\Sigma)$. We shall use the letters $v, w$ to speak about vertices of graphs and $x, y, z$ for elements of the domains or manifold. Several constants will appear, we shall call them $C_1, C_2, \ldots$; each $C_l$ is used exactly once. 
\medskip

\textbf{Plan of the paper.} 
In Sect. \ref{section : 2}, we define precisely what is a triangle-tiling graph. In Sect. \ref{section : 3}, we make the constructions. The leading idea is actually simple: we want to associate a domain to a subgraph. However, we encounter some difficulties for different reasons. One of them is the question of the isolated boundary vertices, also called \textit{bad boundary vertices} in \parencite[Def. $3.1$]{HH}. We solve this problem in Sect. \ref{subsection : domaine rugueux}. Another difficulty comes from the fact that we want the domain to have a smooth boundary. This is the object of Sect. \ref{subsection : domaine lisse}.  In Sect. \ref{section : 4} we prove Theorem \ref{thm : principal}. In order to do so, we want to use Theorem $3$ of \cite{CGR}. Therefore we have to make sure that the hypotheses of the theorem are verified, which is the object of Sect. \ref{subsection : chmt metric}. Once it is done, we apply the theorem and conclude the proof.

\medskip

\textbf{Acknowledgment.} 
I would like to warmly thank my thesis supervisor Bruno Colbois for
having allowed me to work on this subject as well as for his uncountable help and piece of advice which enabled
me to resolve the difficulties encountered. I also wish to thank Niel Smith, Antoine
Gagnebin and the anonymous referees for their careful proofreading of this paper and for their various remarks which have led to its improvement.

\section{Triangle groups and associated triangle-tiling graphs} \label{section : 2}
 
Let us begin by explaining what triangle groups are and what links they have with tessellations of the model spaces $\mathbb{S}^2, \mathbb{E}^2$ and $\H^2$. When it is done, we can explain how to associate a triangle-tiling graph $\Gamma$ to a triangle group.

\begin{defn}
Let $p, q, r \geq 2$ be integers. The triangle group $T^*(p,q,r)$ associated is 
\begin{align*}
T^*(p,q,r)= \langle P, Q, R : P^2=Q^2=R^2=(PQ)^r=(QR)^p=(RP)^q=1 \rangle.
\end{align*} 
\end{defn}
In order to see the links between such an abstract group and a group of reflection, one can think about  $P, Q, R$ as reflections through the opposite sides of a triangle with angles $\frac{\pi}{p}, \frac{\pi}{q}, \frac{\pi}{r}$ respectively.
\medskip

It is well known that a triangle with angles $\alpha, \beta, \gamma$ satisfies $\alpha + \beta + \gamma > \pi$ in the spherical case, while we have  $\alpha + \beta + \gamma = \pi$ in the Euclidean case and that $\alpha + \beta + \gamma < \pi$ in the hyperbolic case. Hence we can regroup the unordered triplets $p,q,r$  according to the value of $\frac{1}{p}+\frac{1}{q}+\frac{1}{r}$. If the number obtained is greater than $1$ we have to think about a spherical triangle, if it is equal to $1$ we have to think about a Euclidean one and if it is less than $1$ we have to think about a hyperbolic one.
\medskip

As said before, we want to work on graphs that have exponential growth rates, therefore we will only consider the third case in this paper. Since one may ask if our result is still true for the two other cases, we remark that in the Euclidean case, the triangle group has polynomial growth rate and then has already been studied in \cite{T}. Regarding the spherical case, the triangle group is finite and hence one can theoretically compute all different possible situations.  
\medskip

Then, from now on, $p, q, r \geq 2$ will be integers satisfying
\begin{align*}
\frac{1}{p}+\frac{1}{q}+\frac{1}{r} <1.
\end{align*} 

\begin{defn} \label{def : hyperbolic plane}
We denote by $\H^2$ the hyperbolic plane, represented here by Poincaré's disk model, which is 
\begin{align*}
\H^2 = \{(x,y) \in \R^2 : x^2+y^2<1\},
\end{align*}
endowed with the Riemannian metric 
\begin{align*}
g(x,y) = 4 \cdot \frac{dx^2 +dy^2}{(1-x^2-y^2)^2}.
\end{align*}

We denote by $d_g( \cdot, \cdot)$ the distance induced by the metric $g$. 
\end{defn}

\begin{rem}
It is a known fact that for any triplet $0 \le \alpha, \beta, \gamma <\pi$ such that $\alpha + \beta + \gamma < \pi$, there exists a hyperbolic triangle with angles $\alpha, \beta, \gamma$. Moreover, there is a unique one up to isometry \parencite[Exercise 7.12]{Bea}.
Hence, given $p, q, r$ as before, there exists a unique triangle which has angles $\frac{\pi}{p}, \frac{\pi}{q}, \frac{\pi}{r}$.
\end{rem}

We state now Theorem $2.8$ of \cite{Mag}:
\begin{thm}
Let $P, Q, R$ be the reflections in the sides of a hyperbolic triangle $\Delta_0$ with angles $\frac{\pi}{p}, \frac{\pi}{q}, \frac{\pi}{r}$. The images of $\Delta_0$ under the action of the distinct elements of the group $T^*(p,q,r)$ generated by $P, Q, R$ fill the hyperbolic plane without gaps and overlapping.
\end{thm}
This means that the choice of the numbers $p, q, r$ gives rise to a tessellation of the hyperbolic plane. Moreover, we know \parencite[Theorem 7.4.1]{Bea} that reflections through geodesics are isometries of $\H^2$. Hence, each tile of the tessellation is a triangle which is isometric to the initial one.

\begin{figure}[H] 
\centering
\includegraphics[scale=0.2]{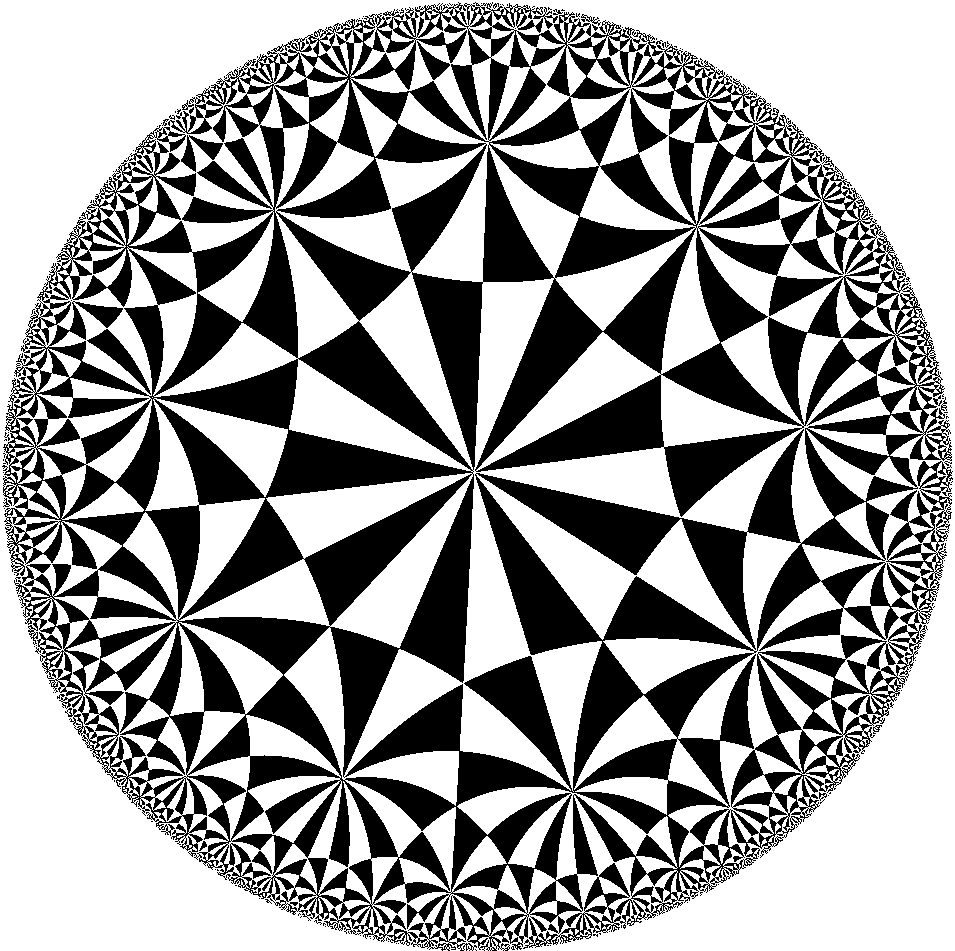}
\caption{Tiling of the hyperbolic plane with congruent triangles of angles $\frac{\pi}{2}, \frac{\pi}{3}$ and $\frac{\pi}{9}$, coming from \cite{Wiki_image}.}
\end{figure}

From such a tiling associated with a triangle group $T^*(p,q,r)$, one can naturally define an infinite simple connected undirected graph $\Gamma = \Gamma(p,q,r)$, called a triangle-tiling graph and that we will use as a host graph. We explain here how to define $\Gamma$.

Each triangle contains a point that is the center of its inscribed circle \parencite[Theorem 7.14.1]{Bea}. We consider these points. They form the set $V$ of vertices of $\Gamma$. The graph structure of $\Gamma$ is defined as follows: two vertices $v_1, v_2 \in V$ are joined by an edge $\{v_1, v_2\}$ if and only if they belong to two adjacent triangles.
\medskip

It is then obvious that $\Gamma = (V, E)$ is an infinite, $3$-regular graph. 
\medskip

We can see $\Gamma$ as a metric space when endowed with the path metric: each edge is of length $1$, the distance between two vertices $v_1, v_2 \in V$ is the minimal number of edges we have to cross to go from $v_1$ to $v_2$.

Because of its links with $\H^2$, it is clear that $\Gamma$ has an exponential growth rate. Hence, as said in Sect. \ref{section : 1}, $\Gamma$ does not enter the class of graphs concerned by Theorem \ref{corollaire : sigma_1 HH}, \ref{corollaire : sigma_1 HP} and~\ref{thm : moi}.
Moreover, $\Gamma$ has cycles, therefore it is not a tree. Hence, it does not enter the class of graphs of Theorem  \ref{thm : HH trees} either.
\medskip

We recall that, given a connected locally finite graph $X$ and any vertex $v$ of $X$, the number of ends of $X$ is $\lim_{n \to \infty} \|X \backslash B(v,n)\|$, where $B(v,n)$ is the ball centered at $v$ with radius $n$ and $\|X \backslash B(v,n)\|$ is the number of infinite connected component of $X \backslash B(v,n)$. It is well known that two roughly isometric graphs have the same number of ends, see  \parencite[Prop. 8.2.8]{Loh}. It is obvious that $\Gamma$ has $1$ end while a Cayley graph of a free group have infinitely many. Therefore, as said before, the structure of $\Gamma$ is completely different from the graphs concerned by Theorem \ref{thm : HH trees} and this difference will be felt in the way we solve the problem.

\begin{defn} \label{def : quasi isom}
A rough isometry between two metric spaces $(X, d_X)$ and $(Y, d_Y)$ is a map $\phi : X \longrightarrow Y$ such that there exist constants $C_1 >1, C_2, \; C_3 >0$ satisfying 
\begin{align*}
C_{1}^{-1} \cdot d_X(x_1,x_2) -C_{2} \leq d_Y(\phi(x_1), \phi(x_2)) \leq C_{1} \cdot d_X(x_1,x_2) + C_{2}
\end{align*}
for all $x_1,x_2 \in X$ and satisfying
\begin{align*}
\bigcup_{x \in X} B(\phi(x), C_{3}) = Y.
\end{align*}
If there is such a map, we say that $X$ is roughly isometric to $Y$.
\end{defn}

\begin{prop}
The host graph $\Gamma$ constructed above is roughly isometric to $(\H^2, g)$, with constants that depend on the value of $p, q, r$.
\end{prop}

\begin{proof}
Take $\phi : \Gamma \longrightarrow \H^2$ as the canonical inclusion and take the constants as the triangle's diameter.

\end{proof}

\medskip

\section{Construction of the domain \texorpdfstring{$N$}{n, sigma}} \label{section : 3}
We consider a finite subset of vertices $\Omega \subset V$, connected for $\Gamma$, giving birth to a subgraph with boundary $\Omega$ as in Definition \ref{def : subgraph}. We recall that each vertex is  the center of a triangle of the tiling and that all triangles are isometric.
\medskip

This section aims to detail a method allowing us to associate a smooth bounded domain $N$ to the subgraph  $\Omega$. The relevance of the domain $N$ lies within its structural links with the subgraph $\Omega$: we will transcribe the structure of $\Omega$ onto $N$. 
\medskip

Before starting, we want to give an overview of the problems that could happen and that we will avoid. 
\medskip

The structural information of $\Omega$ is of two types: the neighborhood structure and the interior/boundary structure. Hence, we have to make sure that the domain of $\H^2$ we will associate to $\Omega$ is able to reflect these two pieces of information. 

In other words, for two $v_1, v_2 \in \bar{\Omega}$, we want $v_1$ to be near $v_2$ in $\Omega$ if and only if $v_1$ is near $v_2$ in the domain. Moreover, for $v \in B$, we want to guarantee the existence of a part of $\Sigma$ near $v$. Reciprocally, for each $x \in \Sigma$, we want to guarantee the existence of a vertex $v \in B$ near $x$. The sense of the word \textit{near} is the following: the proximity between $x$ and $v$ does not depend on the subgraph $\Omega$. This proximity shall be quantified by Proposition \ref{prop : rough isom}.

As already spotted by Han and Hua in \cite{HH}, one of the difficulties comes from the isolated boundary vertices. If $v \in B$ is isolated, we have to be clever to make sure there is $x \in \Sigma$ which is near $v$, see \Cref{expl : sommet isole}.

\medskip

A second difficulty is the following: we want the domain $N$ to be smooth.  This will give us the possibility to make a change of metric on $N$, in order to use Theorem $3$ of \cite{CGR}.
\medskip

Hence the process contains two steps: at first we find a domain $\tilde{N}$ which is structurally related to $\Omega$ but whose boundary $\tilde{\Sigma}$ is not smooth, and secondly we change this domain slightly by smoothing the angles in order to get the wanted domain $N$.

\subsection{Construction of the domain \texorpdfstring{ $\tilde{N}$}{n prime sigma}} \label{subsection : domaine rugueux}

Let us begin by considering a vertex $v \in\bar{\Omega}$ and the triangle $T_v$ associated. In this section, $v$ will always refer to this particular triangle. We call $A_1, A_2, A_3$ the vertices of $T_v$, respectively at angles $\frac{\pi}{p}, \frac{\pi}{q}, \frac{\pi}{r}$. We define a  map $H : \{A_1, A_2, A_3\} \longrightarrow \H^2$
as follows:

$H(A_1)$ is the unique point of the geodesic segment $[v, A_1]$ such that $d_g(v, H(A_1))=\frac{9}{10} \cdot d_g(v, A_1)$. The points $H(A_2)$ and $H(A_3)$ are defined similarly. 
\medskip

We then connect $H(A_1), H(A_2)$ and $H(A_3)$ with geodesic segments. This gives birth to a new triangle, denoted $T_v'$.
By convexity, $T_v'$ is strictly contained inside the initial triangle $T_v$. It is also easy to see that $v$ is contained inside $T_v'$. 
\medskip  

If $w \in \bar{\Omega}$ is another vertex of the subgraph, then by construction there is a triangle $T_w$ of the tiling associated to $w$ and there is an isometry $\psi_{v,w} : \H^2 \longrightarrow \H^2$ such that $\psi_{v,w}(T_v)=T_w$. This isometry is not necessary unique. If there are several, we just pick one and call it $\psi_{v,w}$. We consider this isometry and call $T_w' := \psi_{v, w}(T_v')$.

We apply this process to each vertex of $\bar{\Omega}$. Hence we have now at our disposal $|\bar{\Omega}|$ new triangles, disjoint from each other and isometric to each other.





\medskip

If $v_1, v_2 \in \bar{\Omega}$ are such that $v_1 \sim v_2$ in the subgraph, then by definition of $\Gamma$, $v_1$ and $v_2$ represent the centers of two triangles, let us say $T_1$ and $T_2$, having one side in common. 
Thus $T_1$ has two vertices $x, y$ which are also vertices of the triangle $T_2$. 
As we said before, there is an isometry $\psi_{v, v_1}$ of $\H^2$ such that $\psi_{v, v_1}(T_v)=T_1$. Without loss of generality, say that $\psi_{v, v_1}(A_1)=x$ and $\psi_{v, v_1}(A_2)=y$. 
\medskip

We denote $x_1: =\psi_{v, v_1}(H(A_1))$ and $y_1 := \psi_{v, v_1}(H(A_2))$, which are vertices of the triangle $T_1' = \psi_{v, v_1}(T_v')$. Similarly, we denote $x_2 :=  \psi_{v, v_2}(H(A_1))$ and $y_2 := \psi_{v, v_2}(H(A_2))$ which are vertices of the triangle $T_2' = \psi_{v, v_2}(T_v')$.
\medskip

We then connect $x_1$ to $x_2$ by a geodesic segment, and we do the same with $y_1$ and $y_2$, see Fig. \ref{fig : rectangle}.
\medskip

We write  $T_1' \sim T_2'$ in order to say that we have connected the triangles $T_1'$ and $T_2'$.
\medskip

This process connecting the triangles according to the structure of $\Omega$ allows us to notice the following relation: for two vertices $v_1, v_2 \in \bar{\Omega}$ which are the centers of two triangles $T_1', T_2'$, we have
\begin{align*}
v_1 \sim v_2 \iff T_1' \sim T_2'.
\end{align*}

\begin{figure}[H] 
\centering
\includegraphics[scale=1]{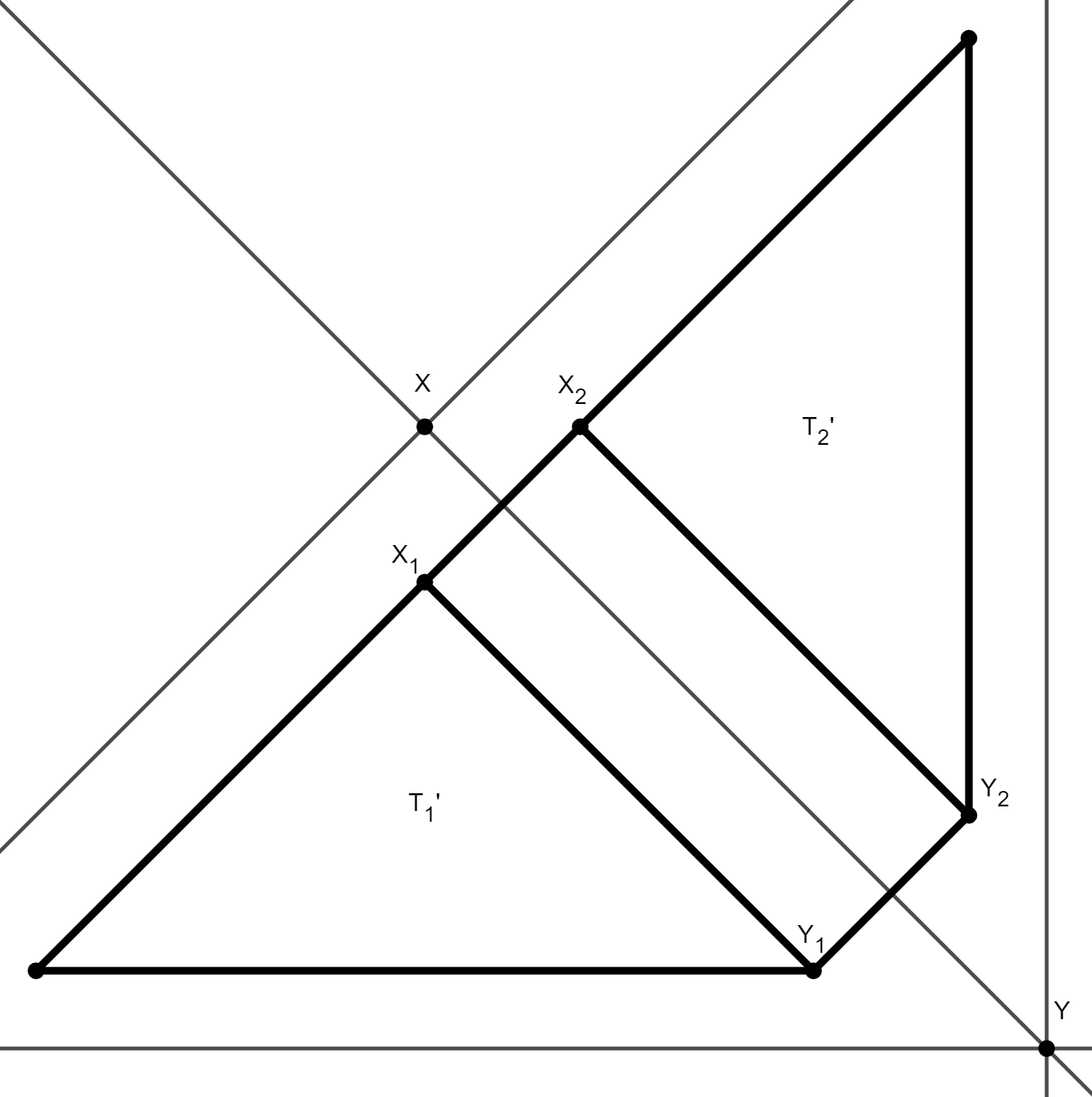}
\caption{The vertices $x_1, y_1$ of $T_1'$ are connected respectively to the vertices $x_2, y_2$ of $T_2'$ because of the assumption that  $v_1 \sim v_2$ in the subraph $\Omega$.}
\label{fig : rectangle}
\end{figure}


Let us suppose that $z$ is the common vertex of $2p$ triangles such that their centers $v_1, \ldots, v_{2p}$ satisfy $v_1 \sim v_2 \sim v_3 \sim \ldots \sim v_{2p} \sim v_1$ in $\Omega$. Without loss of generality, let us say that $\psi_{v, v_1}(A_1)=z$. We denote $z_1=\psi_{v,v_1}(H(A_1)), \ldots, z_{2p}=\psi_{v, v_{2p}}(H(A_1))$ as before.  By applying the process described above, we connect $z_1$ to $z_2$, $z_2$ to $z_3, \ldots, z_{2p}$ to $z_1$ by geodesic segments, see Fig. \ref{fig : triangle}.

Of course, there is nothing specific about $p$ and the same holds for $q$ and $r$.

\begin{figure}[H]
\centering
\includegraphics[scale=1]{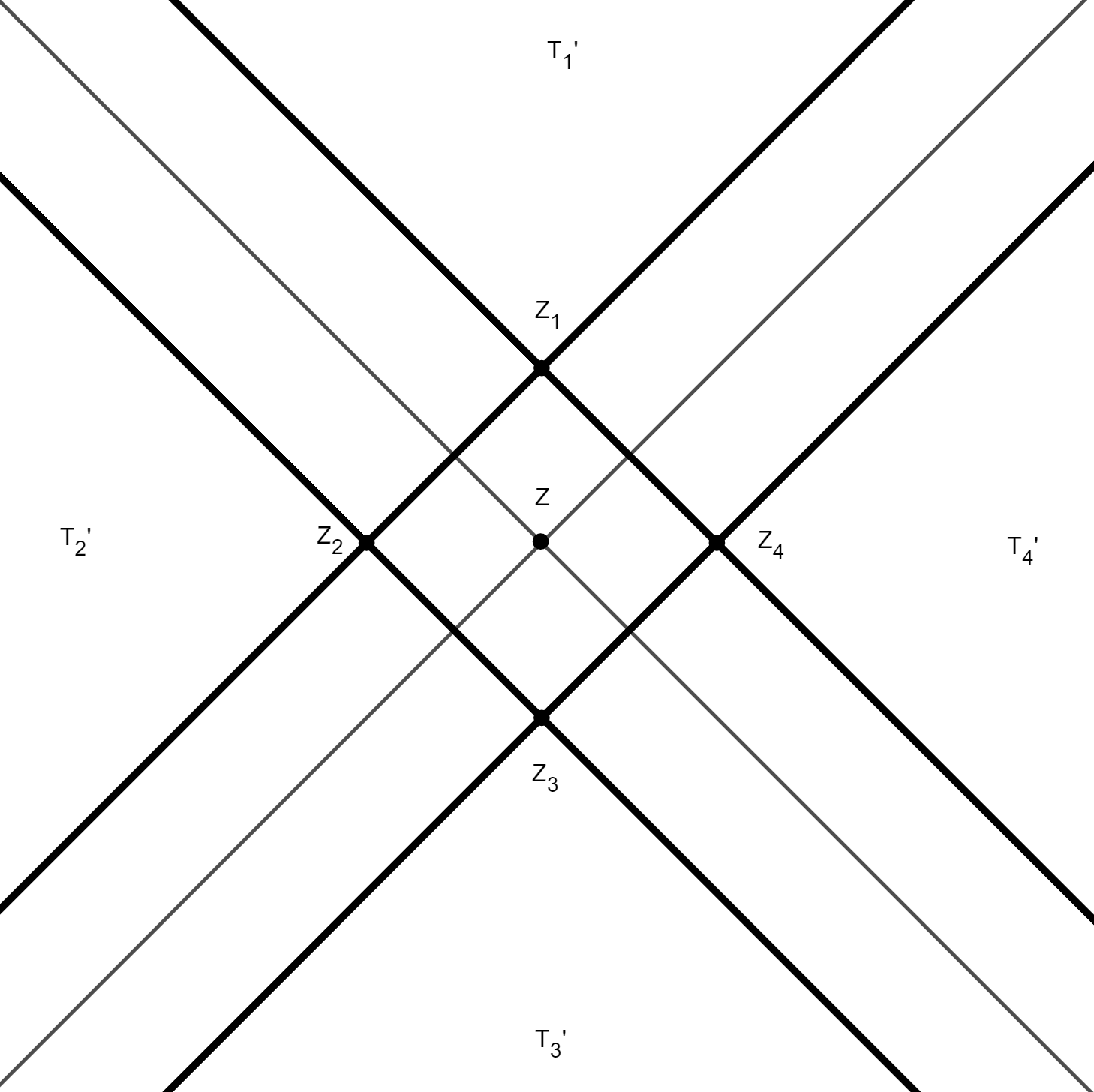}
\caption{We connected $z_1$ to $z_2$, $z_2$ to $z_3$, $z_3$ to $z_4$ and $z_4$ to $z_1$ because of the assumption that $v_1 \sim v_2 \sim v_3 \sim v_4 \sim v_1$ in $\Omega$.}
\label{fig : triangle}
\end{figure}

\begin{rem}
The previous construction naturally generates different \textit{simple polygons} contained inside the hyperbolic plane $\H^2$, of which the exhaustive list is the following:
\begin{itemize}
\item Each vertex $w \in \bar{\Omega}$ adds one triangle $T_w'$;
\item Each couple of vertices $v_1, v_2 \in \bar{\Omega}$ such that $v_1 \sim v_2$ adds one quadrilateral;
\item Each vertex $z$ which is the common vertex of $2p$ triangles such that their centers $v_1, \ldots, v_{2p}$ satisfy $v_1 \sim v_2 \sim v_3 \sim \ldots \sim v_{2p} \sim v_1$ in $\Omega$  adds one $2p$-gon;
\item Each vertex $z$ which is the common vertex of $2q$ triangles such that their centers $v_1, \ldots, v_{2q}$ satisfy $v_1 \sim v_2 \sim v_3 \sim \ldots \sim v_{2q} \sim v_1$ in $\Omega$   adds one $2q$-gon;
\item Each vertex $z$ which is the common vertex of $2r$ triangles such that their centers $v_1, \ldots, v_{2r}$ satisfy $v_1 \sim v_2 \sim v_3 \sim \ldots \sim v_{2r} \sim v_1$ in $\Omega$   adds one $2r$-gon.
\end{itemize}
\end{rem}

\begin{defn}
We call $K$ the compact subset of $\H^2$ obtained by considering the closure of the union of all the simple polygons generated by the previous construction.
We also call $\hat{N}$ the bounded domain of $\H^2$ defined by $\hat{N}= \stackrel{\circ}{K}$ and we call $\hat{\Sigma}$ the boundary of $\hat{N}$.
\end{defn}

\begin{figure}[H] 
\centering
\includegraphics[scale=0.5, angle=0]{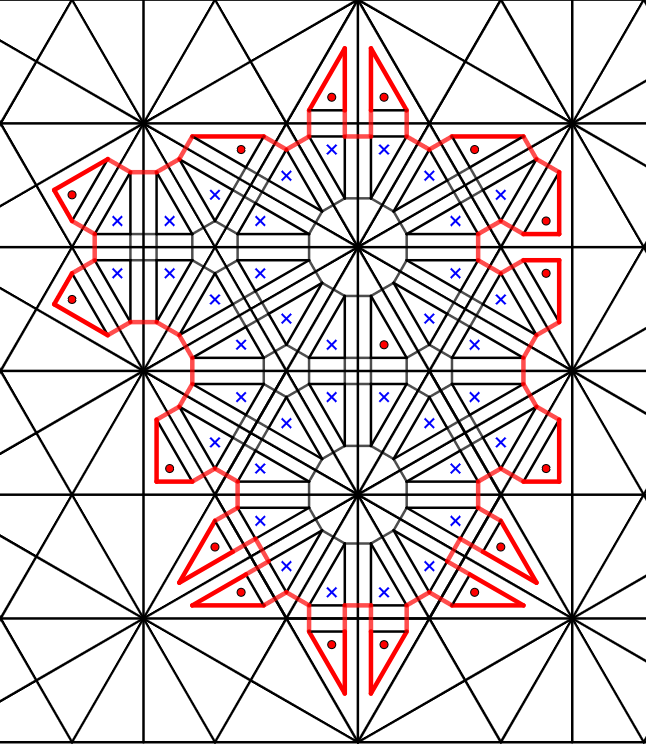}
\caption{The  crosses represent the interior $\Omega$ of the subgraph, the dots represent the boundary $B$ of the subgraph. The polygonal curve  represent the boundary $\hat{\Sigma}$ while the polygon (of which $\hat{\Sigma}$ is the boundary) is the interior $\hat{N}$.}
\label{fig : domaine hat N}
\end{figure}

As this point of the paper, one may ask why we do not simply define the domain as the thickening of the union of all $T_w$, for $w \in \bar{\Omega}$. The reason is that by doing so, the domain would not be able the reflect the neighborhood structure of the subgraph.

Indeed, we recall that by definition of a subgraph, two boundary vertices are never connected by an edge. Let us consider two vertices $w_1, w_2 \in B$ such that $T_{w_1}$ is adjacent to $T_{w_2}$ (meaning that $\{w_1,w_2\} \in E$). Gluing the two triangles $T_{w_1}, T_{w_2}$ would give the information that $w_1$ is adjacent to $w_2$ in the subgraph, which is not the case  because they are two boundary vertices.

This mismatch between the structure of the domain and the structure of the subgraph would then jeopardize one of our next constructions, namely the rough isometry of \Cref{prop : rough isom}. This proposition claims the existence of a rough isometry whose constants do not depend on the subraph $\Omega$ chosen. In order to prove the existence of such a rough isometry, it is crucial that the domain $N$ we are building reflects the neighborhood structure of the subgraph $\Omega$. We give more details about this problem in Appendix \ref{appendix : contre expl}.

\begin{rem}
We recall that, by construction, the domain 
$\hat{N}$ has the same neighborhood structure as the subgraph $\Omega$. Indeed, we already saw that for $v_1, v_2 \in \bar{\Omega}$, 
\begin{align*}
v_1 \sim v_2 \iff T_1' \sim T_2'.
\end{align*}

However, the boundary structure of $\hat{N}$ is not analog to the one of $\Omega$. We already have one implication: for all $x \in \hat{\Sigma}$, there exists $w \in B$ such that $w$ is near  $x$.

The reciprocal is not verified. If $w \in B$, there is no guarantee that there exists $x \in \hat{\Sigma}$ such that  $x$ is near $w$. To see that, one can look at Example \ref{expl : sommet isole}. 
\end{rem}

\begin{expl} \label{expl : sommet isole}
Choose a vertex $v^*$ of the host graph and define $\Omega$ as the ball of radius $n$ deprived of $v^*$. This will give rise to a subgraph $\Omega$, for which $v^* \in B$. However, there is no $x \in \hat{\Sigma}$ near $v^*$. Indeed, the bigger $n$ is, the bigger the distance between $\hat{\Sigma}$ and $v^*$ is. Hence the proximity between $\hat{\Sigma}$ and $v^*$ depends on the subgraph, which we want to avoid. This kind of situation also appears on Fig. \ref{fig : domaine hat N}, where we can see an isolated boundary vertex.
\end{expl}

To remedy this problem, we proceed to a surgery of this domain $\hat{N}$: for each $w \in B$, we remove the ball centered at $w$ of radius $\frac{\rho}{2}$, where $\rho$ denotes the radius of the circle inscribed in $T_w'$, see Fig. \ref{fig : boules enlevees}

\begin{defn}
We call $\tilde{N}$ the domain obtain after the removal of the balls, and we call $\tilde{\Sigma}$ its boundary.
\end{defn}

\begin{rem}
This last surgery obviously gives us the reciprocal we lacked until now: for each $w \in B$, there exists $x \in \tilde{\Sigma}$ such that $x$ is near $w$.
\end{rem}

\begin{figure}[H] 
\centering
\includegraphics[scale=0.45, angle=0]{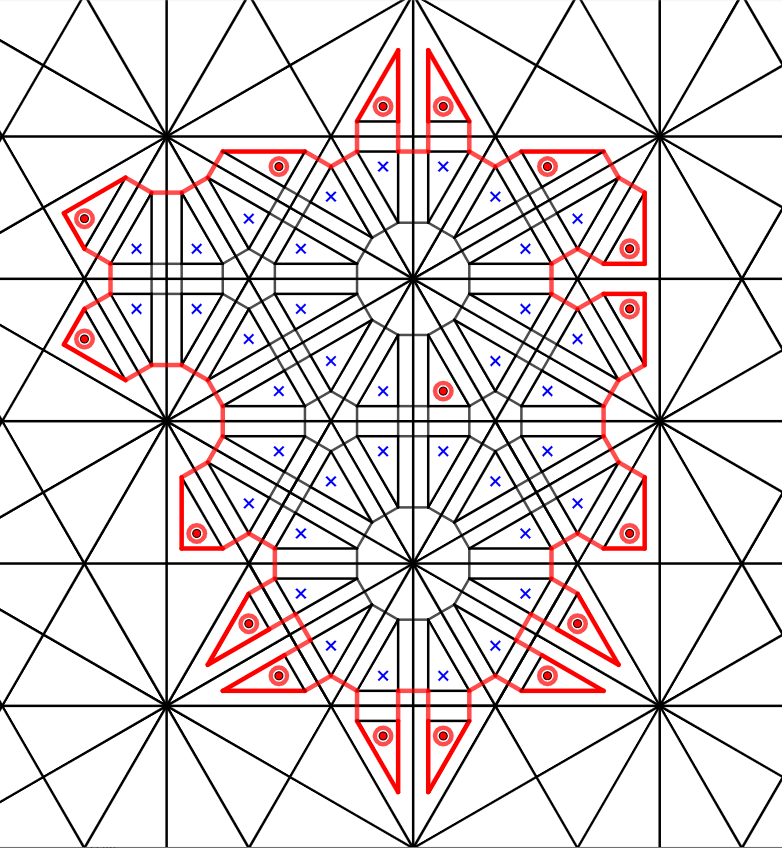}
\caption{The crosses represent the vertices of $\Omega$, the dots represent the boundary $B$. The balls surrounding the boundary vertices are removed from the domain and the structure of the subgraph is readable on the domain. }
\label{fig : boules enlevees}
\end{figure}

This bounded domain $\tilde{N}$ is not our final domain because we want one with a smooth boundary.

\subsection{Smoothing of the domain \texorpdfstring{ $\tilde{N}$}{n prime sigma prime} } \label{subsection : domaine lisse}

As we said in the introduction, we want to discretize the domain in order to find an upper bound for the Steklov spectrum of $\Omega$. 
One way to do this consists in using Theorem $3$ of \cite{CGR}.  Of course, we have to make sure the assumptions of this theorem are verified before using it. However, the domain $\tilde{N}$ does not satisfy all these assumptions, see Remark \ref{rem : cst discretisation}. This section is devoted to modify the domain $\tilde{N}$ and get a new domain $N$ which has the advantage to have a smooth boundary.

\medskip

Note that, as always in this paper, we have to make sure that the operations we make do not depend on the subgraph $\Omega$, but only on the host graph $\Gamma$. 
\medskip

We recall that $\tilde{\Sigma}$ is composed with the union of $\hat{\Sigma}$ and many circles. Each circle is already a smooth connected component of $\tilde{\Sigma}$, hence we only have to smooth $\hat{\Sigma}$ out. Each connected component of $\hat{\Sigma}$ is a simple closed $C^\infty$ piecewise curve, composed with geodesic segments. Note that by construction, there exist at most $4 \times 3 -3 = 9$ different segments (two isometric segments are identified). 
We shall  designate by \textit{corner} the intersection of two geodesic segments forming $\hat{\Sigma}$. A corner is therefore a point  of the curve whose neighborhoods are of class $C^0$, but not of class $C^1$. By construction, a corner is always located on a vertex of a triangle $T'$.
\medskip


The regularity of our construction allows us to state that the domain $\tilde{N}$ has  at most 
$ \left(
\begin{array}{c}
4 \\
2
\end{array}
\right)\times 3 = 18
$ 
different internal angles (two congruent angles are identified).

The interest of these comments is to simplify considerably the smoothing of the domain $\tilde{N}$. Indeed, there are at most $18$ different types of angles to smooth out.
\medskip

Let us call $\lambda_1, \ldots, \lambda_9$ the length of the geodesic segments and let us denote
\begin{align*}
\lambda := \min\{\lambda_1, \ldots, \lambda_9\}.
\end{align*}

If $\tilde{\Sigma}$ has $n$ corners, let us call them $z_1, \ldots z_n$. For each corner $z_i$, there exist exactly two points $x_i, x_i' \in \tilde{\Sigma}$ such that 
\begin{align*}
d_g(x_i, z_i) = d_g(x_i',z_i) = \frac{\lambda}{10}.
\end{align*}

Let us consider a corner $z_i$ as well as the two points $x_i, x_i'$  associated.
\medskip

We then create a smooth curve 
\begin{align*}
\alpha_1 : [0, 1] \longrightarrow \H^2
\end{align*}
such that
\begin{itemize}
\item $\alpha_1(0) = x_i$, $\alpha_1(1)=x_i'$;
\item For all $t \in (0,1)$ we have $\alpha_1(t) \in \tilde{N}$;
\item For all $t \in [0,1]$ we have $d_g(\alpha_1(t), z_i) \leq \frac{\lambda}{10}$;
\item A curve whose image is 
\begin{align*}
[z_{i-1}, x_i] \cup \alpha_1([0,1]) \cup [x_i', z_{i+1}]
\end{align*}
is smooth, see Fig. \ref{fig : lissage coubre 1}.
\end{itemize}

\begin{figure}[H]
\centering
\includegraphics[scale=1.5]{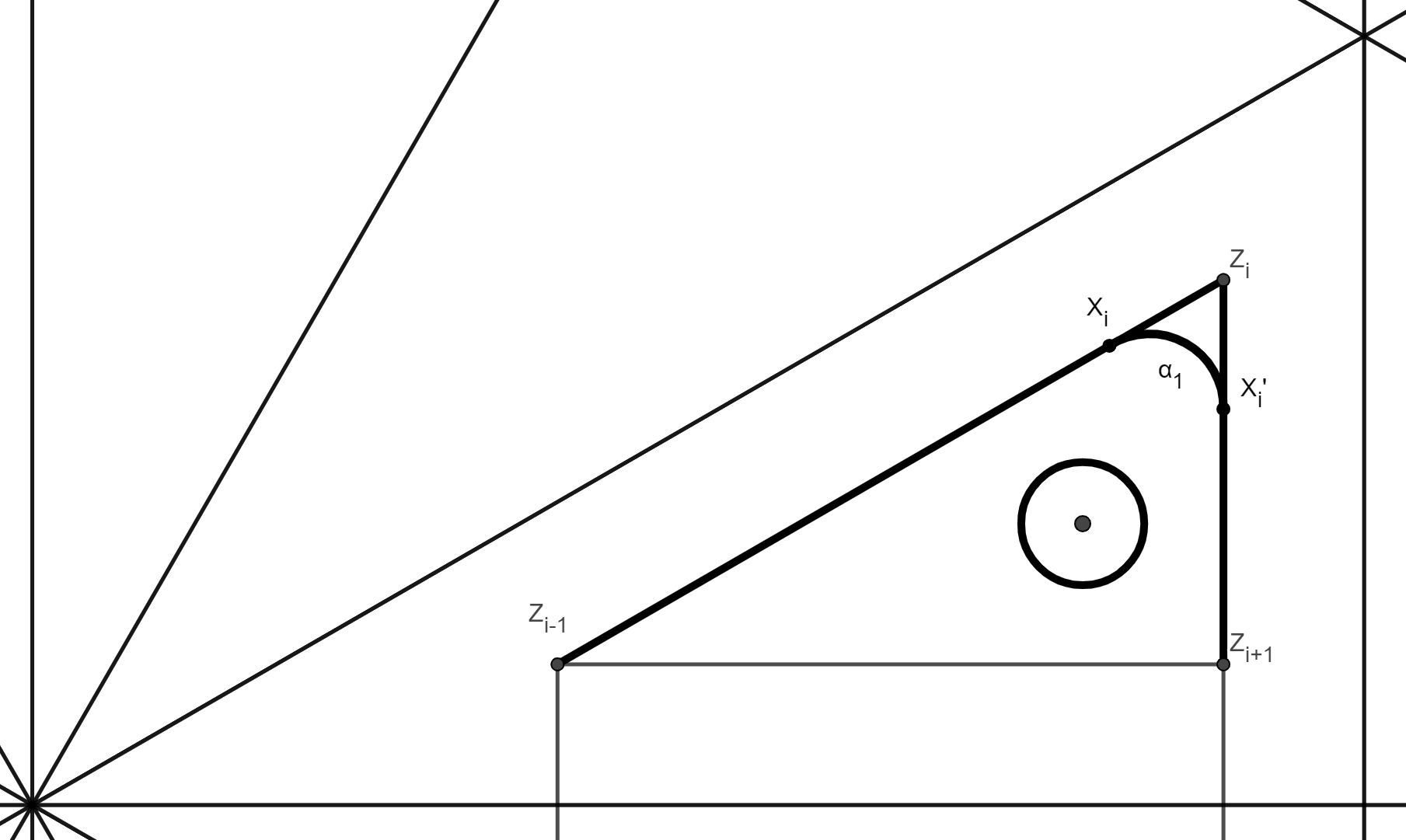}
\caption{The curve $\alpha_1$ can be seen as a smoothing of the angle at the corner $z_i$.}
\label{fig : lissage coubre 1}
\end{figure}

Then suppose that $z_i$ is a corner associated with an angle which is not congruent to the previous one. We then create a smooth curve
\begin{align*}
\alpha_2 : [0,1] \longrightarrow \H^2
\end{align*}
with the same four properties as the previous curve, see Fig. \ref{fig : lissage courbe 2}.
\begin{figure}[H]
\centering
\includegraphics[scale=2.0]{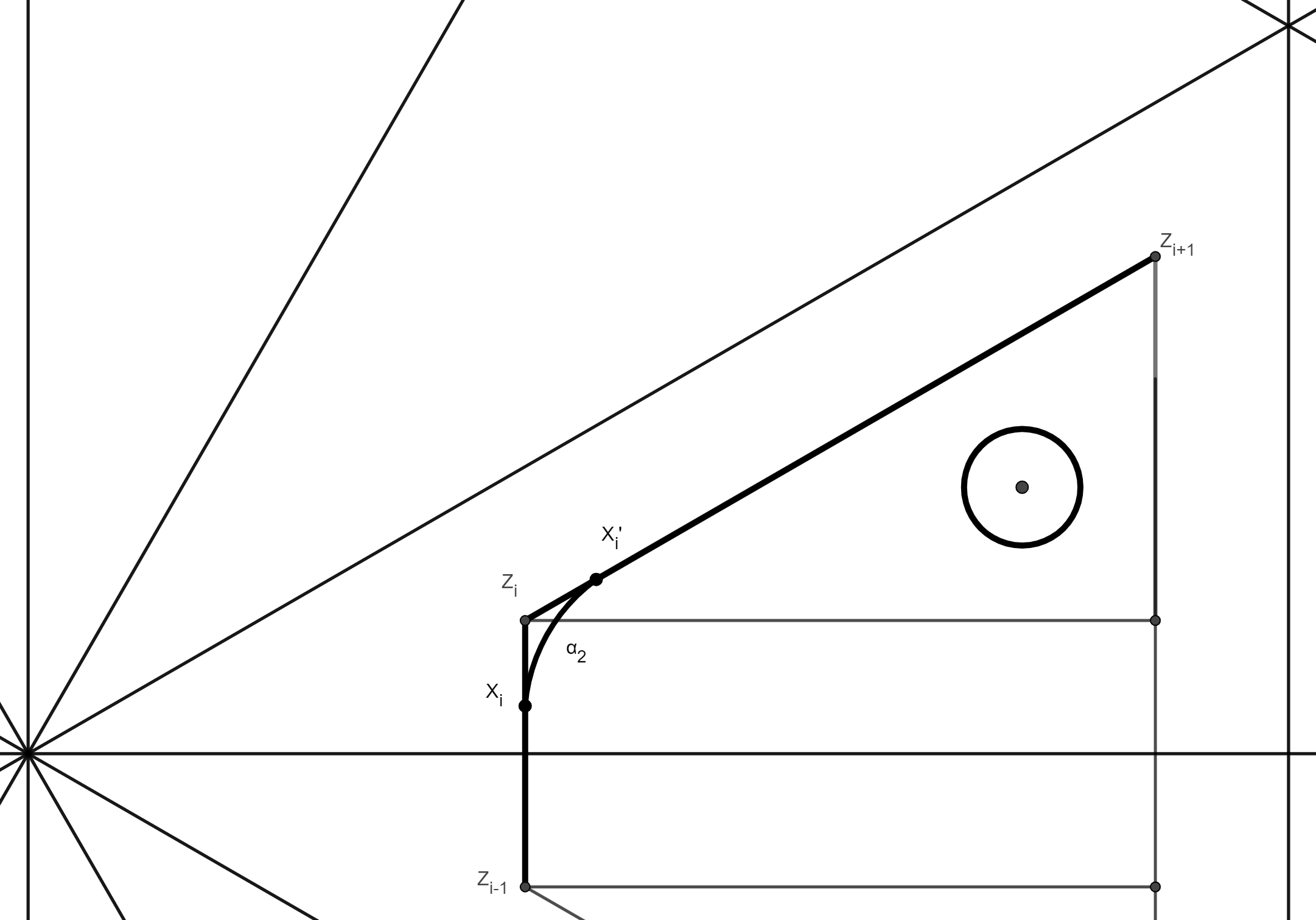}
\caption{The curve $\alpha_2$ is a smoothing of the angle at $z_i$.}
\label{fig : lissage courbe 2}
\end{figure}

We continue the process and create a smoothing curve for each type of angle, at most $18$ times as said before.

\begin{rem}
If $z_j$ is another corner of the same type as $z_i$, meaning that the angle at $z_j$ is congruent to the angle at $z_i$, there is then an isometry $\Psi : \H^2 \longrightarrow \H^2$ which sends the angle at $z_i$ onto the angle at $z_j$. The smoothing curve at angle $z_j$ is then given by $\Psi \circ \alpha_\mu$, where $\mu \in \{1,\ldots, 18\}$ depends on the nature of the angle.
\end{rem}

Thus, we smooth out the domain $\tilde{N}$ with these $18$ curves and obtain a new connected domain with smooth boundary.

We obtain the domain $N$ that we wanted, whose boundary is denoted $\Sigma$. By construction, the domain $N$ has the following characteristics:
\begin{itemize}
\item $N$ is connected;
\item The boundary $\Sigma$ is smooth;
\item $\Sigma$ is composed of at most $28$ types of curve:
\begin{itemize}
\item The $9$ geodesic segments (coming from triangles and quadrilaterals);
\item The $18$ smoothing curves $\alpha_1, \ldots, \alpha_{18}$;
\item The circles resulting from the removal of the balls.
\end{itemize}
\end{itemize}

Moreover, the domain $N$ is constructed in a way that the structure of the subgraph $\Omega$ is readable in it. Indeed, if we call \textit{smoothed triangle} a region of $N$ of the form $N \cap T_w'$, for $w \in \bar{\Omega}$, then
\begin{itemize}
\item A smoothed triangle $N \cap T_{v_1}'$ is connected to a neighbor $N \cap T_{v_2}'$ if and only if $v_1 \sim v_2$ in $\Omega$;
\item A vertex $w$ is part of  $B$ if and only if there exists $x \in \Sigma$ such that $x$ is near $w$. As said before, Proposition \ref{prop : rough isom} will clarify the sense of the word \textit{near}.
\end{itemize}

\begin{figure}[H]
\centering
\includegraphics[scale=0.35, angle=0]{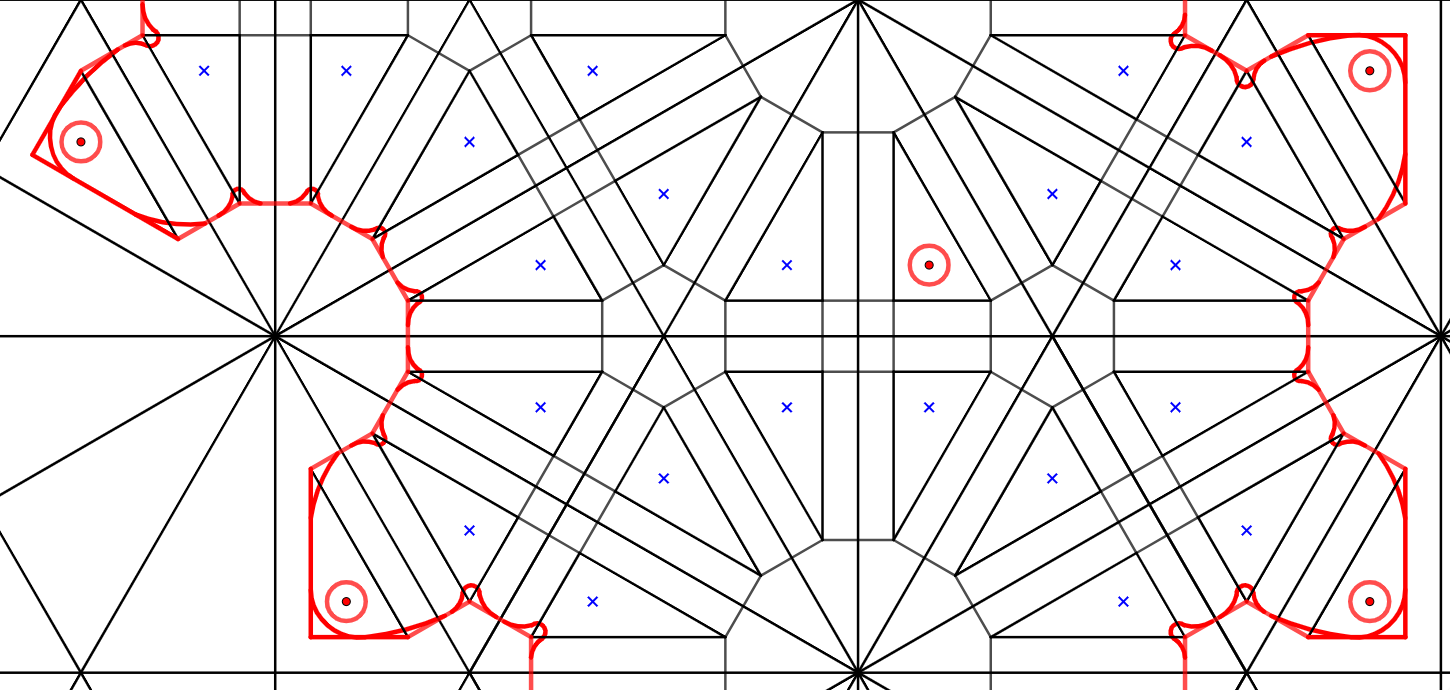}
\caption{The crosses represent the vertices of $\Omega$, the dots represent the boundary $B$. The balls surrounding the boundary vertices are removed from the domain, the structure of the subgraph is readable in the domain and $\Sigma$  is smooth.}
\end{figure}

\begin{rem}
Since each $w \in B$ adds one connected component of $\Sigma$ as a circle, we have the inequality 
\begin{align} \label{ineg : bord}
|\Sigma| \geq C_4 \cdot |B|,
\end{align}
where $C_4$ corresponds to the perimeter of a circle of radius $\frac{\rho}{2}$.
\end{rem}

\section{Proof of the main theorem} \label{section : 4}

Let us begin by recalling Theorem $1.2$ of \cite{CEG1}.
\begin{thm}
There exists a constant $C_5$ such that for all bounded domain $N$ of the hyperbolic space $\H^2$ and for all $k \geq 0$, 
\begin{align} \label{ineg : domaine}
\sigma_k(N, g) \leq C_5 \cdot \frac{k}{|\Sigma|}.
\end{align}
\end{thm}

Actually, the result of Colbois et al. is more general than that, but this statement is enough for our needs.
\medskip

The domain $N$ being structurally similar to the subgraph $\Omega$, we will show that a bound of the same type exists for the subgraph's spectrum. The goal of this section is to transfer this result to the subgraph.
\medskip

To do this, we want to discretize the domain $N$. Let us recall the conditions that the domain must satisfy to be discretized:

We have to assume the existence of constants $\kappa >0$ and $r_0 \in (0,1)$ such that 
\begin{itemize}
\item The boundary $\Sigma$ admits a neighborhood which is isometric to the cylinder $[0,1] \times \Sigma$, whose boundary corresponds to   $\{0\}  \times \Sigma$;
\item The Ricci curvature of $N$ is bounded below by $- \kappa$;
\item The Ricci curvature of $\Sigma$ is bounded below by $0$;
\item For all $x \in N$ such that $d_g(x, \Sigma) >1$, we have inj$_M(x) >r_0$;
\item For all $x \in \Sigma$, we have inj$_\Sigma(x) >r_0$.
\end{itemize}

For further investigation on this topic and to understand why these assumptions are made, one can look at \cite{CGR}.

\begin{rem} \label{rem : cst discretisation}
The last four conditions are trivially satisfied by $N$. Moreover, the constants $\kappa, r_0$ do not depend on the subgraph $\Omega$. Indeed, the regularity of the construction of the domain $N$ allows to give constants $\kappa, r_0$ valid for any domain $N$ obtained by the process described above.

In other words, if we call $\mathcal{M} = \mathcal{M}(\kappa, r_0)$ the class of $2$-dimensional manifolds which satisfy the last four properties, then $N \in \mathcal{M}$ whatever the chosen subgraph $\Omega$.
\end{rem}

On the other hand, the first assumption is not satisfied by the domain. Indeed, $\Sigma$ does not have a neighborhood isometric to a cylinder. To remedy this, we will proceed to a change of metric on $N$ in order to obtain a new Riemannian manifold which satisfies the five properties.

\subsection{Changing the metric on the domain} \label{subsection : chmt metric}

The main difficulty of this subsection is proceeding to a change of metric which is uniform for all domains $N$ obtained by the procedure described in Sect. \ref{section : 3}. 

Here, the word \textit{uniform} reflects the existence of a constant $C_6$ as in Proposition \ref{prop : quasi_isom_continue} which is valid for all domains.
\medskip

Let us denote 
\begin{align*}
N(\delta) = \{x \in N \; : \; d_g(x, \Sigma) \leq \delta \}
\end{align*}
the $\delta$-neighborhood of the boundary.

\begin{prop}[Lemma $34$ of \cite{CGR}]
There exist on $N$ a $\delta >0$ (depending only on the 28 types of curves) and a Riemannian metric $g'$ such that 
\begin{itemize}
\item $(N(\delta), g')$ is isometric to $[0,1] \times \Sigma$;
\item The metrics $g$ and $g'$ are homothetic on $N \backslash N(3\delta)$.
\end{itemize}
\end{prop}

\begin{proof}
We will use the Fermi parallel coordinates: we parametrize each connected component of $\Sigma$ by arc-length and call $s$ the parameter. We then use the distance $t$ to $\Sigma$ as a second parameter to describe the points of $N$ lying in a close neighborhood of $\Sigma$. In these coordinates, the hyperbolic metric is expressed by
\begin{align*}
g(s,t) = \varphi(s,t) \cdot ds^2 +  dt^2,
\end{align*}
where $\varphi$ is a smooth positive function satisfying $\varphi(s, 0) =1$.

Let $\delta > 0$ be small enough to have $\frac{1}{2} \leq \varphi(s,t) \leq 2$ on $N(3\delta)$ (such a $\delta$ exists because $\varphi$ is smooth).

We call $g_0$ the product metric which, in the Fermi coordinates $(s,t)$, is expressed by 
\begin{align*}
g_0(s,t) = ds^2 +dt^2.
\end{align*}
We then take a smooth function 
\begin{align*}
\chi : [0, 3 \delta] \longrightarrow [0,1]
\end{align*}
such that $\chi \equiv 0$ on $[0, \delta]$, $\chi \equiv 1$ on $[2 \delta, 3\delta]$ and such that $\chi$ is strictly increasing on $[\delta, 2 \delta]$.

Then we define the metric 
\begin{align*}
g_\delta(s,t) = \chi(t) g(s,t) + (1-\chi(t))g_0(s,t).
\end{align*}

This metric coincides with the hyperbolic metric on $N(3\delta)\backslash N(2\delta)$, then it can be extended all over the domain $N$ into a metric that we continue to call $g_\delta$.

Moreover, endowed with this metric, $N(\delta)$ is isometric to $[0, \delta] \times \Sigma$. We then define the metric 
\begin{align*}
g' := \frac{1}{\delta^2}g_\delta,
\end{align*}
for the cylindrical neighborhood to have length $1$.

\end{proof}

The value of $\delta$ depends only on the 28 types of curves composing $\Sigma$. That is the reason we built the domain $N$ with such regularity. Thanks to the process, we can choose $\delta$ independently of the subgraph $\Omega$ chosen.

\begin{prop} [Lemma $34$ of \cite{CGR}] \label{prop : quasi_isom_continue}
There exists a constant $C_6 >1$,  that does not depend on the subgraph $\Omega$, such that for all $x \in N$ and all $v \in T_xN, \; v \neq 0$, we have
\begin{align*}
\frac{1}{C_6} \leq \frac{g'(x)(v,v)}{g(x)(v,v)} \leq C_6.
\end{align*}
\end{prop}

\begin{proof}
We distinguish three cases:
\begin{itemize}
\item $x \in N\backslash N(2\delta)$;
\item $x \in N(\delta)$;
\item $x \in N(2\delta)\backslash N(\delta)$.
\end{itemize}
Let us start with the first one. Let $x \in N\backslash N(2\delta)$ and $0 \ne v \in T_xN$. We have
\begin{align*}
\frac{g'(x)(v,v)}{g(x)(v,v)} = \frac{\frac{1}{\delta^2}g_\delta(x)(v,v)}{g(x)(v,v)} = \frac{\frac{1}{\delta^2}g(x)(v,v)}{g(x)(v,v)} = \frac{1}{\delta^2}
\end{align*}
because on $N \backslash N(2\delta)$, the metric $g_\delta$ coincides with the hyperbolic metric $g$.

For the second case, let $x \in N(\delta)$ and $0 \ne v \in T_xN$.  we have 
\begin{align*}
\frac{g'(x)(v,v)}{g(x)(v,v)} & = \frac{g'(x)(v,v)}{(\varphi(s,t)ds^2 + dt^2)(v,v)} 
                              \leq \frac{g'(x)(v,v)}{(\frac{1}{2}ds^2 + \frac{1}{2}dt^2)(v,v)} \\
                             & = \frac{\frac{1}{\delta^2}g_\delta(x)(v,v)}{\frac{1}{2}(ds^2+dt^2)(v,v)} 
                              = \frac{\frac{1}{\delta^2}g_0(x)(v,v)}{\frac{1}{2}g_0(x)(v,v)} \\
                             & = \frac{2}{\delta^2}                             
\end{align*}
because $g_\delta$ coincides with the product metric $g_0$ on $N(\delta)$.

In a similar way, we have
\begin{align*}
\frac{g'(x)(v,v)}{g(x)(v,v)} & = \frac{g'(x)(v,v)}{(\varphi(s,t)ds^2 + dt^2)(v,v)} 
                              \geq \frac{g'(x)(v,v)}{(2ds^2 + 2dt^2)(v,v)} \\
                             & = \frac{\frac{1}{\delta^2}g_\delta(x)(v,v)}{2(ds^2+dt^2)(v,v)} 
                              = \frac{\frac{1}{\delta^2}g_0(x)(v,v)}{2g_0(x)(v,v)} \\
                             & = \frac{1}{2\delta^2}.                           
\end{align*}
Let us now look at the third case. Let $x \in N(2\delta)\backslash N(\delta)$ and $0 \ne v \in T_xN$. 

We recall that on $N(2\delta)\backslash N(\delta)$, the metric $g_\delta$ interpolates the product metric $g_0$ and the hyperbolic metric $g$ with the help of a smooth increasing function  $\chi$.

Then we have
\begin{align*}
\frac{g'(x)(v,v)}{g(x)(v,v)} & = \frac{\frac{1}{\delta^2}g_\delta(x)(v,v)}{g(x)(v,v)} 
                              = \frac{\frac{1}{\delta^2}(\chi(t)g(s,t) +(1-\chi(t))g_0(s,t))(v,v)}{g(x)(v,v)} \\
                             & = \frac{1}{\delta^2} \left( \chi(t) + (1-\chi(t))\frac{g_0(s,t)(v,v)}{g(x)(v,v)} \right) 
                              \geq \frac{1}{\delta^2} \left( \chi(t) + (1-\chi(t))\frac{g_0(s,t)(v,v)}{2 g_0(x)(v,v)} \right) \\
                             & = \frac{\chi(t)}{\delta^2} + \frac{1-\chi(t)}{2\delta^2} 
                              \geq \frac{1}{2 \delta^2}.
\end{align*}
Similarly, we have
\begin{align*}
\frac{g'(x)(v,v)}{g(x)(v,v)} & = \frac{\frac{1}{\delta^2}g_\delta(x)(v,v)}{g(x)(v,v)} 
                              = \frac{\frac{1}{\delta^2}(\chi(t)g(s,t) +(1-\chi(t))g_0(s,t))(v,v)}{g(x)(v,v)} \\
                             & = \frac{1}{\delta^2} \left( \chi(t) + (1-\chi(t))\frac{g_0(s,t)(v,v)}{g(x)(v,v)} \right) 
                              \leq \frac{1}{\delta^2} \left( \chi(t) + (1-\chi(t))\frac{g_0(s,t)(v,v)}{\frac{1}{2} g_0(x)(v,v)} \right) \\
                             & = \frac{\chi(t)}{\delta^2} + \frac{1-\chi(t)}{\frac{1}{2}\delta^2} \\
                             & \leq \frac{1}{\frac{1}{2} \delta^2} 
                              = \frac{2}{\delta^2}.
\end{align*}
Then the ratio is bounded for all $x \in N$ and for all $v \in T_xN, \; v \neq 0$, and we can choose 
\begin{align*}
C_6 :=   \frac{2}{ \delta^2}.
\end{align*}

Moreover, this constant $C_6$ does not depend on the chosen subgraph $\Omega$. Indeed, the function $\varphi$ depends only on the, at most, $28$ types of curves forming $\Sigma$ 
(which we have fixed once and for all), and $\delta$ depends only on $\varphi$. Thus, as said before, the constant $\delta >0$ can be chosen independently of the subgraph, which allows us to fix a universal value of $C_6 >1$ for all the domains $N$ obtained thanks to the procedure described in Sect. \ref{section : 3}.

\end{proof}

We now have at our disposal a new Riemannian manifold with boundary, denoted $(N, g')$, which is related to $(N, g)$ in the sense of Proposition \ref{prop : quasi_isom_continue}. We recall now Proposition $32$ of~\cite{CGR}:
\begin{prop} 
Let $N$ be a Riemannian manifold of dimension $m$, compact with smooth boundary and let $g, g'$ be two Riemannian metrics on $N$. Let us assume that there exists a constant $C_6 >1$ such that for all $x \in N$ and for all $v \in T_xN, \; v \neq 0$, we have
\begin{align*}
\frac{1}{C_6} \leq \frac{g'(x)(v,v)}{g(x)(v,v)} \leq C_6.
\end{align*}
Then we have 
\begin{align*}
\frac{1}{C_6^{2m+1}} \leq \frac{\sigma_k(N,g')}{\sigma_k(N, g)} \leq C_6^{2m+1}.
\end{align*}
\end{prop}

The assumption is exactly what we prove at Proposition \ref{prop : quasi_isom_continue}. Hence we can apply this result to $(N, g)$ and $(N, g')$ in order to get:
\begin{align} \label{ineg : metrique}
\sigma_k(N, g') \leq C_6^5 \cdot \sigma_k(N, g).
\end{align}

\subsection{Discretization of the manifold \texorpdfstring{ $(N, g')$}{n sigma g prime}} \label{section : 5}

Let us recall that we proceeded to a change of metric on  $N$ in order to give it the ability to be discretized, according to constants $r_0$ and $\kappa$, as said in Remark \ref{rem : cst discretisation}. There exist several ways to discretize a manifold. In this paper, we apply the process described in \cite{CGR}, for we want the discretization to have a spectral link with the manifold. 
\medskip

This process is the following:

We choose $\epsilon \in (0, r_0/4)$ and we choose $V_\Sigma$ a maximal $\epsilon$-separated subset of $\Sigma$. Then we call $V_\Sigma'$ the copy of $V_\Sigma$ lying $4\epsilon$ away from the boundary:
\begin{align*}
V_\Sigma' = \{4\epsilon\} \times V_\Sigma.
\end{align*}
Then we choose $V_I$ a maximal $\epsilon$-separated subset of $N \backslash [0,4\epsilon] \times \Sigma$ such that $V_\Sigma' \subset V_I$.

Then we consider the subset $\tilde{V} = V_\Sigma \cup V_I$ and grant it the structure of a graph by decreeing
\begin{itemize}
\item Two vertices $v, w \in \tilde{V}$ are adjacents as soon as $d_{g'}(v, w) \leq 3\epsilon$;
\item A vertex $v \in V_\Sigma$ is adjacent to its counterpart $v' \in V_\Sigma'$.
\end{itemize}
This process gives a graph with boundary $(\tilde{V}, \tilde{E}, V_\Sigma)$, simply denoted $(\tilde{V},V_\Sigma)$ hereafter, whose boundary is $V_\Sigma$ and  that we call  $\epsilon$-discretization of $N$.

\medskip

Theorem $3$ point $4)$ of \cite{CGR} allows us to state:
\begin{thm}
There exists a constant  $C_7 >0$ depending only on $\kappa, r_0$ and $\epsilon$ such that for all $k \leq |V_\Sigma|$, we have
\begin{align} \label{ineg : discretisation}
\sigma_k(\tilde{V}, V_\Sigma) \leq C_7 \cdot \sigma_k(N, g') \cdot k.
\end{align}
\end{thm}

\subsection{Rough isometry between \texorpdfstring{$(\tilde{V}, V_\Sigma)$}{discrétisation} and \texorpdfstring{$\Omega$}{initial}} \label{section : 6}

We now want to exploit the graph $(\tilde{V}, V_\Sigma)$ for which we have an upper bound relative to its spectrum to control the spectrum of our initial subgraph $\Omega$. In order to do it, we will have to deal with the concept of rough isometry once again. This will allow us to use Proposition $16$ of \cite{CGR} to compare the Steklov spectra of the graphs. The main difficulty here is that we have to make sure the constants of the rough isometry are independant of the subgraph $\Omega$. Let us begin by defining what is a rough isometry in the context of graphs with boundary.

\begin{defn}
A rough isometry $\phi$ between two graphs with boundary $(\bar{\Omega}_1, E'_1, B_1)$ and $(\bar{\Omega}_2, E'_2, B_2)$ is a rough isometry which sends  $B_1$ onto $B_2$ and such that the restriction of $\phi$ to $B_1$ is a rough isometry $B_1 \longrightarrow B_2$ when
considering extrinsic distances on $B_1$ and $B_2$.
\end{defn}

\begin{prop} \label{prop : rough isom}
There exists a rough isometry  $\bar{\phi} : (\tilde{V}, V_\Sigma) \longrightarrow \bar{\Omega}$ whose constants $C_1, C_2, C_3$ are independent from the subgraph $\Omega$.
\end{prop}

\begin{proof}
We have to define a map $\bar{\phi} : (\tilde{V}, V_\Sigma) \longrightarrow \bar{\Omega}$ and show that it is a rough isometry. 

Remark that the vertices $v$ of $\tilde{V}$ can be of different types. There are boundary vertices coming from the $28$ different kind of curves forming $\Sigma$, and  there are interior vertices coming from $N$.
As a consequence, the definition of  $\bar{\phi}$ is a little bit heavy, but the idea to define the rough isometry is very natural: each vertex $v \in \tilde{V}$ is sent onto the vertex $w$ of $\bar{\Omega}$ which is of same nature (interior or boundary) and which is the nearest from it.
\medskip

Let us define 
\begin{align*}
\bar{\phi} : (\tilde{V}, V_\Sigma) \longrightarrow \bar{\Omega}.
\end{align*}
For the vertices of the boundary:
\begin{itemize}
\item For $v \in V_\Sigma$ such that $v$ is part of a side of a triangle $T'$, we choose $\bar{\phi}(v) \in B$ the vertex at the center of $T'$;
\item For $v \in V_\Sigma$ such that $v$ is part of the boundary of a ball that had been removed, we choose $\bar{\phi}(v) \in B$ the vertex at the center of the removed ball;
\item For $v \in V_\Sigma$ such that $v$ is part of a side of a quadrilateral, we find the side of a triangle closest to $v$ and we choose $\bar{\phi}(v) \in B$ as if $v$ were on this triangle's side;
\item For $v \in V_\Sigma$ such that $v$ is part of a smoothing curve, we find the side of a triangle closest to $v$ and choose $\bar{\phi}(v) \in B$ as if $v$ were on this triangle's side.
\end{itemize}

\begin{figure}[H]
\centering
\includegraphics[scale=1.3]{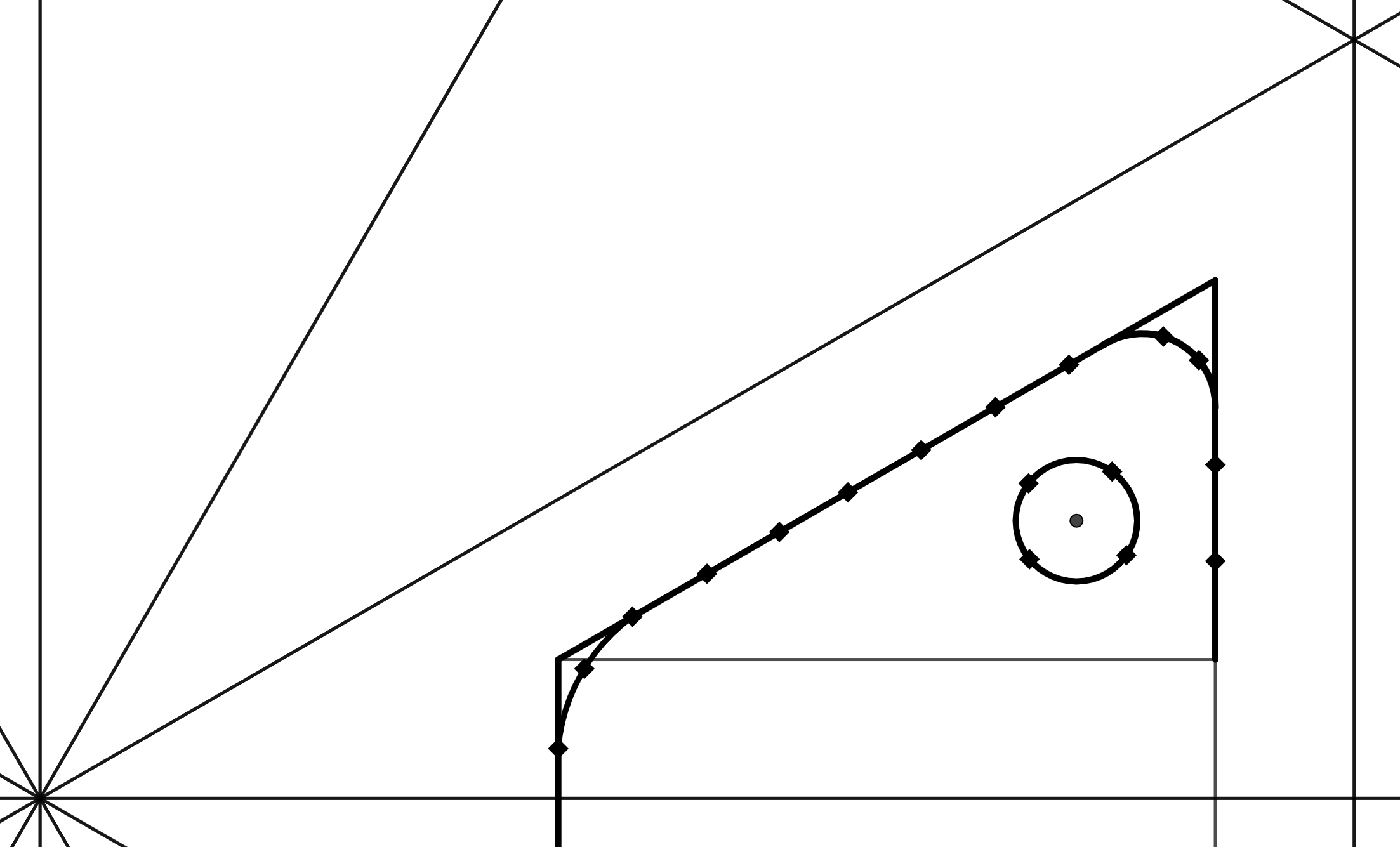}
\caption{The vertices  of $V_\Sigma$ are represented by diamonds, the dot vertex belongs to $B$. All of the diamonds are sent to the dot by $\bar{\phi}$.}
\end{figure}

And for the interior vertices:
\begin{itemize}
\item For $v \in V_I$ such that $v$ is part of a triangle whose center is $w \in \Omega$, we choose $\bar{\phi}(v) = w$;
\item For $v \in V_I$ such that $v$ is part of a triangle whose center is $w \in B$, then there exists at least one $w' \in \Omega$ such that $w \sim w'$. We then choose $\bar{\phi}(v) = w'$. If there are several possibilities, we choose one once and for all;
\item For $v \in V_I$ such that $v$ is part of a quadrilateral, then two opposite sides of this quadrilateral are the sides of two triangles $T_1', T_2'$. At least one of them has a center $w \in \Omega$. We then choose $\bar{\phi}(v) = w$.  If there are two possibilities, we choose one once and for all;
\item For $v \in V_I$ such that $v$ is part of a $2p$-gon (respectively $2q$-gon, $2r$-gon), then this $2p$-gon (resp. $2q$-gon, $2r$-gon) is surrounded by $2p$ (resp. $2q, 2r$) triangles $T_1', \ldots, T_{2p}'$ (resp. $T_{2q}', T_{2r}'$) of which at least $p$ (resp. $q, r$) have a center $w \in \Omega$. We then choose $\bar{\phi}(v) = w$ once and for all.
\end{itemize}

\begin{figure}[H]
\centering
\includegraphics[scale=1.2]{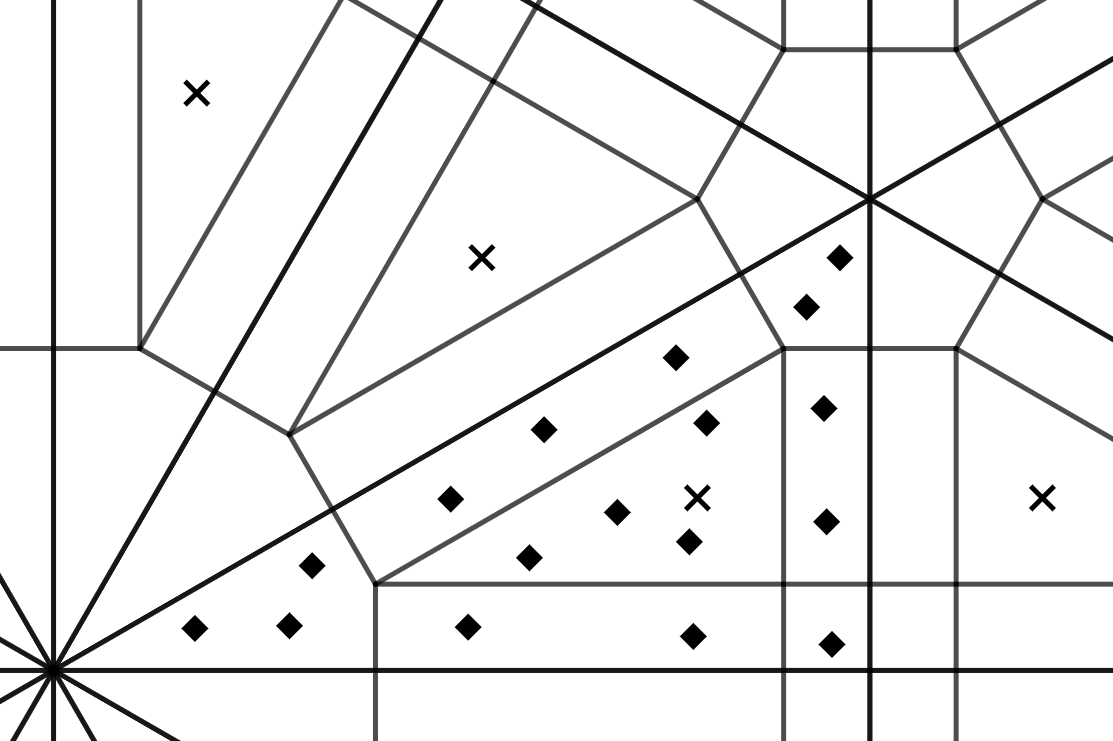}
\caption{The diamond vertices are part of $V_I$,  the cross vertices belongs to $\Omega$. All of the diamonds vertices are send to the bottom left cross vertex by $\bar{\phi}$.}
\end{figure}

In order to show that  $\bar{\phi}$ is a rough isometry, 
let us partition the domain $N$ into cobblestones: a cobblestone $\mathcal{C}$ is defined as the intersection of a triangle $T$ of the initial tiling with $N$. 
If $w \in \bar{\Omega}$ is the center of a triangle $T_w$, we denote by $\mathcal{C}_w$ the associated cobblestone. We also write $\mathcal{C}_w \sim \mathcal{C}_{w'}$ to say that two cobblestones are adjacent.
\medskip

Then we choose $C_1$ as the cardinality of the biggest possible $\epsilon$-separated set contained inside a cobblestone multiplied by $\max\{p, q, r\}$. Then we choose $C_2 = C_1$. Thus, if two vertices $v_1, v_2 \in \tilde{V}$ belongs to the same cobblestone, we have $d_{\tilde{V}}(v_1, v_2) \leq C_1$.
\medskip

We recall that by our construction of the domain $N$, for $w, w' \in \bar{\Omega}$ we have 
$$
w \sim w' \iff \mathcal{C}_w \sim \mathcal{C}_{w'},
$$
i.e the neighborhood structure of the subgraph is readable onto the domain. Therefore, for $w_1, w_2 \in \bar{\Omega}, \; w_1 \neq w_2$, the distance $d_{\bar{\Omega}}(w_1, w_2)$ represents the number of cobblestones that separate $w_1$ from $w_2$ plus one. Thus, if $v_1, v_2 \in \tilde{V}$ are such that $\bar{\phi}(v_1) = w_1$ and $\bar{\phi}(v_2)= w_2$, then we have 
\begin{align*}
C_1^{-1} d_{\tilde{V}}(v_1, v_2) -C_2 \leq d_{\bar{\Omega}}(w_1, w_2) \leq C_1 d_{\tilde{V}}(v_1, v_2) +C_2.
\end{align*}
Moreover, $\bar{\phi}$ is a surjective map so we can choose $C_3 = 1$ and we get
\begin{align*}
\bigcup_{v \in \tilde{V}} B(\bar{\phi}(v), C_3) = \bar{\Omega}.
\end{align*}

\end{proof}

We can now recall Proposition $16$ of \cite{CGR}:
\begin{prop}
Given $C_1 \geq 1, C_2, C_3 \geq 0$, there exist some constants $C_8, C_9$ depending only on $C_1, C_2, C_3$ and of the maximal degree of the vertices such that for all graphs with boundary $(\Gamma_1, B_1), (\Gamma_2, B_2)$ roughly isometric with constants $C_1, C_2, C_3$, we have
\begin{align*}
C_8 \leq \frac{\sigma_k(\Gamma_1, B_1)}{\sigma_k(\Gamma_2, B_2)} \leq C_9.
\end{align*}
\end{prop}

Applied to this situation, we obtain 
\begin{align} \label{ineg : isom}
\sigma_k(\Omega) \leq \frac{1}{C_8} \sigma_k(\tilde{V}, V_\Sigma).
\end{align}

\subsection{Conclusion} \label{section : 7}
In this section, we prove Theorem \ref{thm : principal} and Corollary \ref{cor : zero}.

\begin{proof}
Throughout the paper, we got different results, that we can now assemble to finally obtain \Cref{thm : principal}:
\begin{align*}
\sigma_k(\Omega)&  \stackrel{(\ref{ineg : isom})}{\leq} \frac{1}{C_8} \cdot \sigma_k(\tilde{V}, V_\Sigma) \\
                   &  \stackrel{(\ref{ineg : discretisation})}{\leq} \frac{1}{C_8} \cdot C_7 \cdot \sigma_k(N, g') \cdot k \\
                   &  \stackrel{(\ref{ineg : metrique})}{\leq} \frac{1}{C_8} \cdot C_7 \cdot C_6^5 \cdot \sigma_k(N, g) \cdot k \\
                   &  \stackrel{(\ref{ineg : domaine})}{\leq} \frac{1}{C_8} \cdot C_7 \cdot C_6^5 \cdot C_6 \cdot \frac{k}{|\Sigma|} \cdot k \\
                   & \stackrel{(\ref{ineg : bord})}{\leq} \frac{1}{C_8} \cdot C_7 \cdot C_6^5 \cdot C_6 \cdot \frac{k}{C_4 \cdot |B|} \cdot k \\
                   & =: C \cdot \frac{1}{|B|} \cdot k^2.
\end{align*}

All along the paper, we took care of specifying on which parameters the constants depend. It happens that they do not depend on the subgraph $\Omega$ chosen. They only depend on the host graph $\Gamma$ and on $\epsilon$. Therefore, if we set a value for $\epsilon$, we can take the same constant $C$ for all subgraph $\Omega$ of $\Gamma$; it is now fixed once and for all.
\medskip

As a consequence, for a choice of three integers $p, q, r \geq 2$ such that $\frac{1}{p} + \frac{1}{q}+ \frac{1}{r} < 1$, giving birth to a tessellation of the hyperbolic plane and to a host graph $\Gamma$ as defined in Sect. \ref{section : 2}, there exists a constant $C = C(\Gamma)$ such that for any subgraph $\Omega$ of $\Gamma$, we have
\begin{align*}
\sigma_k(\Omega) \leq C(\Gamma) \cdot \frac{1}{|B|} \cdot k^2.
\end{align*}

\end{proof}

From this statement, let us prove Corollary \ref{cor : zero}.
\begin{proof}
It is enough to notice the following fact: for $(\Omega_l, B_l)_{l \geq 1}$ a family of subgraphs of $\Gamma$ such that $|\Omega_l| \longrightarrow \infty$, then we also have $|B_l| \longrightarrow \infty$. 
\medskip

Therefore, for all  $k \in \N$ fixed, we have
\begin{align*}
\sigma_k(\Omega_l, B_l) \leq C(\Gamma) \cdot \frac{1}{|B_l|} \cdot k^2 \underset{l \to \infty}{ \longrightarrow} 0.
\end{align*}

\end{proof}

\section{Consideration and interrogation} \label{sect : interrogation}

All the constructions above were about a host graph $\Gamma$, which was a triangle-tiling graph. However, one may have noticed that we could have used other polygons rather than triangles and still obtained the result. The information we used is the finite number of possible situations, like the $28$ different kinds of curves composing $\Sigma$ or the $18$ types of angles to smooth out.

All these constructions could have emerged from any exact tessellation of the hyperbolic plane, as long as the tiles are compact and the number of different polygon in the tessellation is finite (the tessellation is exact if and only if each edge of a tile is an edge of exactly two polygons of the tessellation). If we used other polygons rather than triangles, the number of different possible situations would have been larger, and the constants would have been different. Nevertheless, the result would have been the same.
\medskip

This comment shows that the result we get in this paper is more general than it primarily seems. Unfortunately, it has its limits. If we get interested in a tiling of the hyperbolic plane which has infinitely many kinds of tiles, then our construction is not relevant anymore. In the same way, if a tile of the tessellation is not compact, we cannot use our method either.
\medskip

This consideration leads to an open question:
\begin{question} \label{quest : autre graphe}
If $\Gamma$ is any graph roughly isometric to the hyperbolic plane, is there a constant $C = C(\Gamma)$ such that a bound as in Theorem $\ref{thm : principal}$ exists?
\end{question}

This question naturally leads to a more general interrogation. In order to properly define the problem, let us give a definition.

\begin{defn}
We say that a host graph $\Gamma$ has the property (P) if for each $k \in \N$ and each family $(\Omega_l)_{l \ge 1}$ of subgraphs of $\Gamma$, we have 
\begin{align*}
     |\Omega_l| \underset{l \to \infty}{\longrightarrow} \infty \implies \sigma_k(\Omega_l) \underset{l \to \infty}{\longrightarrow} 0.
\end{align*}
\end{defn}

Now we can ask the following open question:
\begin{question} \label{question : comportement}
Let $\Gamma_1, \Gamma_2$ be two roughly isometric graphs. Let us assume that $\Gamma_1$ has the property (P). Does $\Gamma_2$ also have the property (P)? 
\end{question}

Reformulated in the language of geometric group theory, the question becomes 
\begin{center}
\textit{Is the property (P) a large scale invariant?}
\end{center}

This question, apparently not so hard, appears to be more thorny  than expected. 

If positively answered, it would automatically generalise our result to any graph roughly isometric to the hyperbolic plane, and it would certainly have many other applications.
\medskip

Another interesting interrogation one may have consists in wondering if some similar constructions could be done in the hyperbolic space $\H^n$, with $n \ge3$. In particular, a first question is the following:
\begin{center}
    \textit{Is there a natural class of graphs, analogous to triangle-tiling graphs, that would be roughly
isometric to $\H^n$?}
\end{center}
The answer to this question is \textit{yes.} Using \parencite[Sect. 6.8]{Rat}, we can generate  tessellations of $\H^n$ with polyhedra, for any $n \ge 2$. From such a tessellation, we can define a host graph $\Gamma$ in the same manner as we did in this paper. It could be interesting to study such a host graph and see if some results analogous to \Cref{thm : principal} hold in higher dimension. This consideration leads to the following open question:
\begin{question}
Let $\Gamma$ be a graph coming from a  polyhedral tessellation of $\H^n, n \ge 3$. Does a constant $C=C(\Gamma)$ exist, such that a bound as in \Cref{thm : principal} holds?
\end{question}

\begin{appendices}

\section{About the importance of the small triangles in our construction } \label{appendix : contre expl}

We provide here an example which shows that, given a subgraph $\Omega$ of $\Gamma$, we cannot simply consider the domain that we get when thickening the union of $T_w$ for all $w\in \bar{\Omega}$. 
\medskip

Let us consider the subgraph given by the following figure:
\begin{figure}[H]
    \centering
    \includegraphics[scale=0.5]{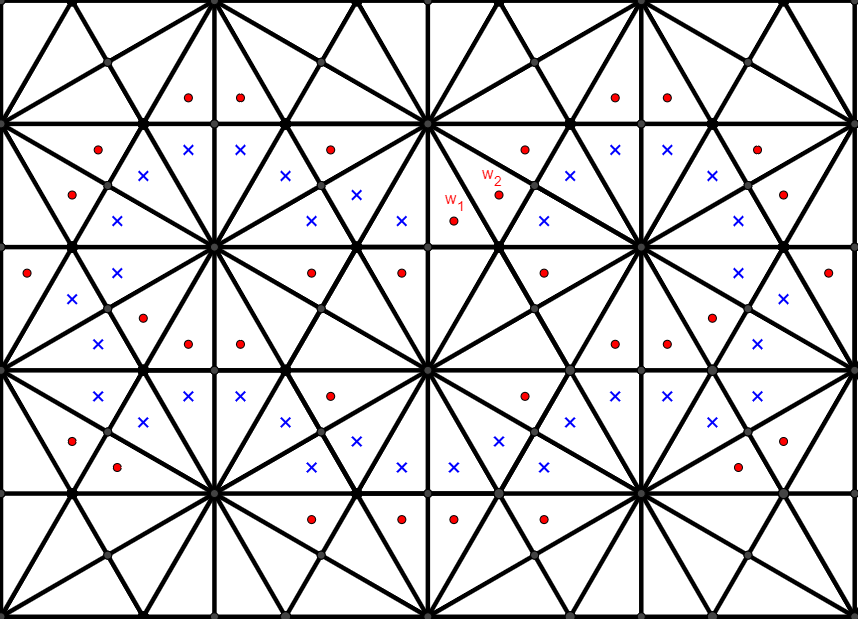}
    \caption{The crosses vertices form the interior of the subgraph, the dot vertices form the boundary.}
    \label{fig : graphe contre expl}
\end{figure}
We are particularly interested in the boundary vertices named $w_1$ and $w_2$ in Fig. \ref{fig : graphe contre expl}. Here are two properties that $w_1$ and $w_2$ have:
\begin{itemize}
    \item $w_1$ is close to $w_2$ in the host graph. Indeed, they belong to two adjacent triangles of the tessellation. Therefore, $d_\Gamma(w_1, w_2)=1$ (where we used the notation $d_\Gamma$ for the distance in the host graph).
    \item $w_1$ is far from $w_2$ in the subgraph. Indeed, by definition there is no edge between $w_1$ and $w_2$ in the subgraph. In fact, we have $d_\Omega(w_1, w_2)=33$, which is the diameter of the subgraph (we used the notation $d_\Omega$ for the distance in the subgraph).
\end{itemize}
Because we are building a domain which is a sort of analog of the subgraph, we have to make sure that the distance between $w_1$ and $w_2$ is large in the domain.

The domain $\hat{N}$ that we get from this subgraph, using the strategy  presented in this paper (using the small triangles), is the following:
\begin{figure}[H]
    \centering
    \includegraphics[scale=0.4]{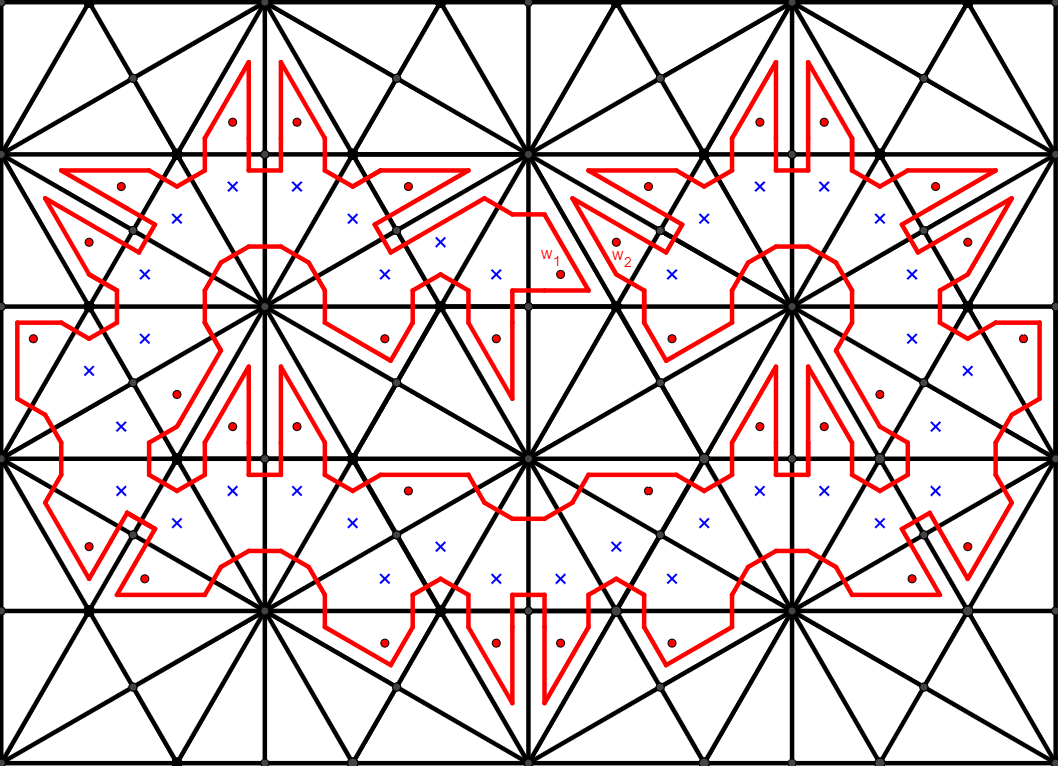}
    \caption{Using now $d_{\hat{N}}$ as a notation for the distance in $\hat{N}$, we can easily see that $d_{\hat{N}}(w_1, w_2)$ is large, roughly as the diameter of $\hat{N}$.}
    \label{fig: domaine contre expl correct}
\end{figure}

Here is now the domain that we get while considering the union of triangle $T_w$ for all $w\in \bar{\Omega}$:
\begin{figure}[H]
    \centering
    \includegraphics[scale=0.46]{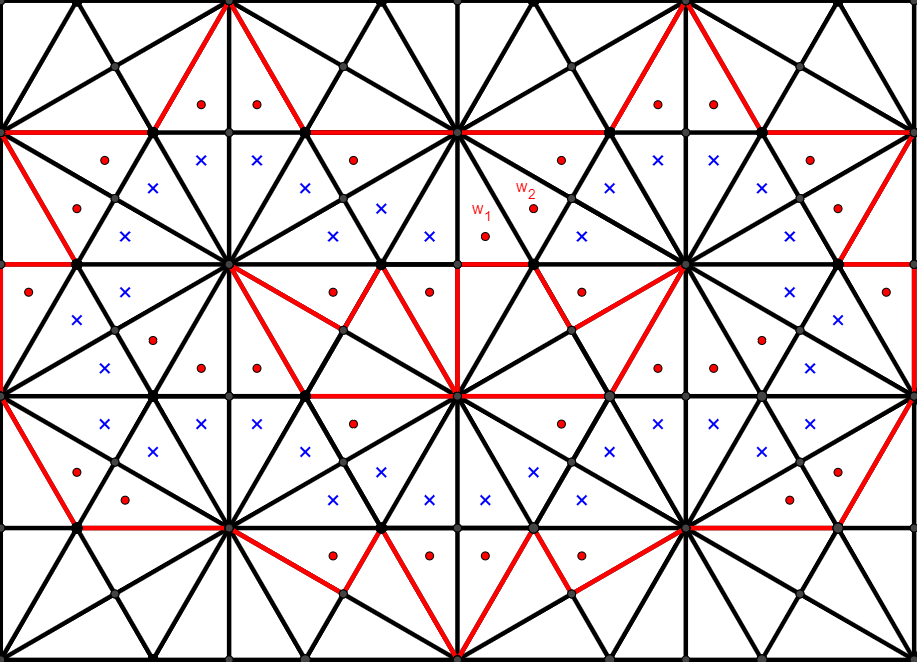}
    \caption{We can see that the distance in the domain between $w_1$ and $w_2$ is small. }
    \label{fig: domaine contre expl incorrect}
\end{figure}

If we were to pursue our construction with the domain given by Fig. \ref{fig: domaine contre expl incorrect}, we would have a real problem when building the rough isometry of Proposition \ref{prop : rough isom}. 
\medskip

Indeed, let us now consider a family of subgraphs $(\Omega_l)_{l \ge 1}$, such that $|\Omega| \underset{l\to \infty}{ \longrightarrow} \infty$ and such that each subgraph of the family has the same particular property as the subgraph of Fig. \ref{fig : graphe contre expl} (the property concerning $w_1$ and $w_2$ we discussed above). In that case, the constants in the rough isometry would then have to be chosen according to each subgraph (the diameter of each subgraph would do). This would obviously destroy our result.

\end{appendices}

\printbibliography

\medskip



\end{document}